\documentclass[8pt,twoside]{article}
\usepackage [english]{babel}
\usepackage[ansinew]{inputenc}
\usepackage{amsfonts}
\usepackage{amsthm}
\usepackage{amsmath}
\usepackage{hyperref}
\usepackage{authblk}
\usepackage{amssymb, mathtools,xspace}
\usepackage[hmargin=1.5cm, vmargin=1.5cm]{geometry}
\usepackage{amssymb}
\usepackage [T1]{fontenc} 
\usepackage{cases}
\usepackage{pdfpages}
\usepackage{float}
\usepackage[percent]{overpic}
\usepackage{multicol}
\usepackage{subfloat}
\usepackage{graphics}
\usepackage{subcaption}
\usepackage{dblfloatfix}
\usepackage{fixltx2e}
\newtheorem{theorem}{Theorem}[section]

\newtheorem{remark}{Remark}[section]

\newtheorem{lemma}{Lemma}[section]
\newtheorem{proposition}{Proposition}[section]

\begin{document}
\title{ \textbf{Null Controllability of a nonlinear age, space and two-sex structured population dynamics model } }
\author{Yacouba Simpor\'e \footnote{Laboratoire LAMI, Universit\'e  Joseph Ki-Zerbo,01 BP 7021 Ouaga 01, Burkina Faso. DeustoTech, Fundaci\'on Deusto Avda. Universidades, 24, 48007, Bilbao, Basque Country, Spain (\texttt{simplesaint@gmail.com})}~\footnote{\textbf{Funding}: This project has received funding from the European Research Council (ERC) under the European Union's Horizon 2020 research and innovation programme (grant agreement No. 694126-Dycon).
 }\and Oumar Traor\'e\footnote{Laboratoire LAMI, Universit\'e  Joseph Ki-Zerbo,01 BP 7021 Ouaga 01, Burkina Faso (\texttt{otraoret@gmail.com})}}
\date{}
\maketitle
\begin{abstract}
	In this  paper, we study the null controllability of a nonlinear age, space and two-sex structured population dynamics model. This model is such that the nonlinearity and the couplage are at birth level.\\
	We consider a population with males and females and we are dealing with two cases of null controllability problem.\\ The first problem is related to the total extinction, which means that, we estimate a time $T$ to bring the male and female subpopulation density to zero.\\ The second case concerns null controllability of male or female subpopulation individuals. Here, So if $A$ is the life span of the individuals, at time $T+A$ one will get the total extinction of the population. \\
	Our method uses first an observability inequality related to the adjoint of an auxiliary system, a null controllability of the linear auxiliary system, and after the Schauder's fixed point theorem.
		\end{abstract}
		\textbf{Key words}: two-sex population dynamics model, Null controllability, method of characteristics, Observability inequality, Schauder fixed point.\\
		\textbf{AMS subject classifications.} 93B03, 93B05, 92D25	
\section{Introduction and Main results}
In this paper, we study the null-controllability of an infinite dimensional nonlinear  coupled system describing the dynamics of two-sex structured population with spatial position.\\
Let $(m,f)$ be the solution of the following system:
\begin{equation}
\left\{ 
\begin{array}{ccc}
m_t+m_a-K_m\Delta m+\mu_m m&=\chi_{\Xi}v_m&\text{ in }Q ,\\ 
f_t+f_a-K_f\Delta f+\mu_f f&=\chi_{\Xi'}v_f&\text{ in }Q,\\ 
m(\sigma,a,t)=f(\sigma,a,t)&=0&\text{ on }\Sigma, \\
m(x,a,0)=m_0\quad f(x,a,0)=f_0&&\text{ in }Q_A,\\
m(x,0,t)=(1-\gamma)N(x,t),\quad f(x,0,t)=\gamma N(x,t)&& \text{ in }Q_T,\\
N(x,t)=\int_{0}^{A}\beta(a,M)fda\text{; } M=\int_{0}^{A}\lambda(a)mda && \text{ in }Q_T.
\end{array}
\right.
\label{1}
\end{equation}
where $T$ is a positive number, $\gamma\in]0,1[,$ and $\Omega$ a bounded open subset of $\mathbb{R}^{N}$ whose boundary is assumed to be of class $C^2$. Here $m_{z}$ and $f_{z}$ stand respectively for differentiation of $m$ and $f$ with respect to the variable $z$ with $z\in \{a,t\}.$\\
In $(\ref{1})$,  $Q=\Omega\times (0,A)\times (0,T)\text{; } \Xi=\omega\times (a_1,a_2)\times (0,T)\text{; } \Xi'=\omega'\times (b_1,b_2)\times (0,T)$ where $0\leq a_1<a_2\leq A$, \\$0\leq b_1<b_2\leq A$, $\omega\text{ and } \omega'$ two nonempty open subsets of $\Omega$;  $Q_A=\Omega\times (0,A)$, $Q_T=\Omega\times (0,T)$ and \\$\Sigma=\partial\Omega\times (0,A)\times (0,T).$\\
The functions, $m(x,a,t)$ and $f(x,a,t)$ describe  respectively the density of males and females of age $a$ being at time $t$  at the location $x$.\\
Moreover, $\mu_m$ and $ \mu_f$ denote respectively the natural mortality rate of males and females. The control functions are $v_m$ and $v_f$, and depend on $x$, $a$ and $t$. In addition, $\chi_{G}$ denoted the characteristic function of the set $G$. \\
We have denoted by $\beta$ the positive function describing the birth rate that depends on $a$ and also on \[M=\int_{0}^{A}\lambda (a)mda\] where $\lambda$ is a fertility function of the male individuals. Thus the densities of newly born male and female individuals at location $x$ and at time $t$ are given respectively by $m(x,0,t)=(1-\gamma)N(x,t)$ and  $ f(x,0,t)=\gamma N(x,t)$ where
\[N(x,t)=\int_{0}^{A}\beta(a,M)f(x,a,t)da.\]
We assume that the fertility rates $\beta$, $\lambda$ and the mortality rates $\mu_f,$ $\mu_m$ satisfy the demographic properties:
\[
(H_1):
\left\{
\begin{array}{c}
\mu_m\geq 0, \quad \mu_f\geq 0\text{ a.e in } [0,A], \\\
\mu_m\in L_{loc}^{1}\left([0,A)\right),\quad \mu_f\in L_{loc}^{1}\left([0,A)\right), \\
\int_{0}^{A}\mu_m(a)da=+\infty,\quad \int_{0}^{A}\mu_f(a)da=+\infty
\end{array} 
\right..
\]
,
\[
(H_{2}):
\left\{
\begin{array}{c}
\beta\in C\left([0,A]\times \mathbb{R}\right)\\
\beta(a,p)\geq 0\text{ for every } (a,p)\in[0,A]\times \mathbb{R}\\
\beta(a,0)=0\text{ in } (0,A).
\end{array}
\right.
\]
The assumption $\beta(a,0)=0\text{ for } a\in (0,A)$ means that, the birth rate is zero if there are not male individuals.\\
We further assume that the birth function $\beta$ verifies the following assumption:
\[(H_{3}):\left\{
\begin{array}{c}
(i)-\text{There exists }b\in (0,A)\text{ such that }\beta(a,p)=0, \forall a\in (0,b),\\
 (ii)-\text{ there exists } L>0 \text{ such that }|\beta(a,p)-\beta(a,q)|\leq L|p-q| \text{ for all } p, q \in\mathbb{R}, \text{ a.e }a\in (0,A)\\
	(iii)-\text{there exists }\|\beta\|_{\infty}>0 \text{ such that } 0\leq\beta(a,p)\leq \|\beta\|_{\infty}. 
\end{array}
	\right.\]
	We suppose also that
\[
(H_{4}):
\left\{
\begin{array}{c}
\lambda\in C^{1}\left([0,A]\right),\\
\lambda\geq 0\text{ for every } a\in [0,A].
\end{array}
\right..
\]	
\[
(H_5):\left\{
\begin{array}{cc}
\lambda\mu_m\in L^{1}(0,A).\\
\end{array}
\right.
\]	
Finally concerning the initial data, we suppose that \[(m_0,f_0)\in \left(L^{\infty}(Q_A)\right)^2 \text{ and } m_0\text{, } f_0 \geq 0\text{ a.e in } Q_A.\]
With these ingredients in hand, we can state our main results.
  \begin{theorem}
Let us assume the assumptions $(H_1)-(H_2)-(H_3)-(H_4)-(H_5).$ If $a_1<b$ and $(a_1,a_2)\subset (b_1,b_2)$ for every time $T>a_1+A-a_2,$ and for every $(m_0,f_0)\in \left(L^2(Q_A)\right)^2,$ there exists $(v_m,v_f)\in L^2(\Xi)\times L^2(\Xi')$ such that the associated solution $(m,f)$ of the system $(\ref{1})$ verifies:
	\begin{equation}
	m(x,a,T)=0 \quad a.e\quad x\in \Omega\quad a\in (0,A),\label{rev}\end{equation}
	\begin{equation}
	f(x,a,T)=0 \quad a.e\quad x\in \Omega\quad a\in (0,A).\label{rev}\end{equation}
\end{theorem}
This result shows that one can act on the males in a locality $D_1$ and on the females in the locality $D_2$ ($D_1\cap D_2$ can be empty) to get the extinction of the total population at the final time $T>a_1+A-a_2.$
\begin{theorem} 
Let us assume the assumptions $(H_1)-(H_2)-(H_3)-(H_4)-(H_5).$\\ We  have:
\begin{itemize}
	\item[-(1)] Let $v_f=0$. For any $\varrho>0,$ for every time $T>A-a_2$ and for every $(m_0,f_0)\in \left(L^2(Q_A)\right)^2,$there exists a control $v_m\in L^2(\Theta)$ such that the associated solution $(m,f)$ of the system $(\ref{1})$ verifies:
	\begin{equation}m(x,a,T)=0 \quad a.e\quad x\in \Omega\quad a\in (\varrho,A)
	\end{equation} where $ \Theta =\omega\times(0,a_2)\times (0,T).$
	\item[-(2)] Let $v_m=0$. For every time $T>a_1+A-a_2$ and for every $(m_0,f_0)\in \left(L^2(Q_A)\right)^2,$ there exists a control $v_f\in L^2(\Xi)$ such that the associated solution $(m,f)$ of the system $(\ref{1})$ verifies:
	\begin{equation}f(x,a,T)=0 \quad a.e\quad x\in \Omega\quad a\in (0,A).	
	\label{rev}
	\end{equation}
	\end{itemize}
	\end{theorem}
	\begin{remark}
		In fact, the assumption $H_3-(i),$ is not a mandatory condition for the proof of Theorem 1.2-(1).
			\end{remark}
	In practice this study takes place for example in the fight against malaria.
Malaria is a serious disease (In 2017 there were 219 million cases in the World $\cite{report}$)and our work takes its importance in the strategy to fight against it.\\
In fact Malaria is a vector-borne disease transmitted by an infective female anopheles mosquito.
A malaria control strategy in Brazil or Burkina Faso consists of releasing genetically modified male mosquitoes (precisely sterile males)  in the nature in order to reduce reproduction since females mate only once in their life cycle.\\
In the theoretical framework, very few authors have studied control problems of two-sex structured population dynamics model.\\
The control problems of coupled systems of population dynamics models take an intense interest and are widely investigated in many papers. Among them we have $\cite{1}$, $\cite{2}$, $\cite{fine}$ and the references therein. In fact, in $\cite{1}$ the authors studied a coupled reaction-diffusion equations describing interaction between a prey population and a predator one. The goal of this work is to look for a suitable control supported on a small spatial subdomain which guarantees the stabilization of the predator population to zero. In $\cite{2}$, the objective was different. More precisely, the authors consider an age-dependent prey-predator system and they prove the existence and uniqueness of an optimal control (called also "optimal effort") which gives the maximal harvest via the study of the optimal harvesting problem associated to their coupled model.\\
In $\cite{yu}$ He and Ainseba study the null controllability of a butterfly population by acting on eggs, larvas and female moths in a small age interval.\\
In $\cite{fine},$ the authors analyze the growth of a two-sex population with a fixed age-specific sex ratio without diffusion. The model is intended to give an insight into the dynamics of a population where the mating process takes place at random choice and the proportion between females and males is not influenced by environmental or social factors, but only depends on a differential mortality or on a possible transition from one sex to the other (e.g. in sequential hermaphrodite species).
In $\cite{fine}$ first, a basic model asymptotically linear, is considered and its ergodicity is studied. Survival thresholds and their dependence on the sex ratio are then analysed, in connection with the optimal sex ratio to guarantee survival.
A further model including logistic effect is also considered and discussed in connection with existence and stability of steady states.\\ 
In this work, we study the null controllability of a nonlinear two-sex population dynamics structured model with diffusion. Unlike the model treated in $\cite{fine},$ we consider a nonlinear cascade system with  two different fertility rates: the fertility rate of the male $\lambda$ and the fertility rate $\beta$ of the female that depend on the total population of the fertile males. It should be noted that the mortality and the diffusion coefficient also depend on the sex of the individual.\\ We use  the technique of $(\cite{dy})$ combining final-state observability estimates with the use of characteristics to establish the observability inequalities necessary for the null controllability property of the auxiliary systems. Roughly, in our method we first study an approximate null controllability result for an auxiliary cascade system. Afterwards, we prove the null controllability result for the system $(\ref{1})$ by means of Schauder's fixed point theorem.\\
The remainder of this paper is as follows:  In Section 2 we study an approximate null controllability result for some associated auxiliary model. Section 3 is devoted to the proof of the Theorem 1.1. Next, we prove the Theorem 1.2 in section 4.
 	  	  \section{Approximate null controllability of an auxiliary coupled system}
 	  	  In this section we study of an auxiliary system obtained from the system $(\ref{1}).$\\
 	  	  Let $p$ be a $L^{2}(Q_T)$ function, we define the auxiliary system given by:
\begin{equation}
\left\{ 
\begin{array}{ccc}
m_t+m_a-K_m\Delta m+\mu_m m&=\chi_{\Xi}v&\text{ in }Q ,\\ 
f_t+f_a-K_f\Delta f+\mu_f f&=\chi_{\Xi'}u&\text{ in }Q,\\ 
m(\sigma,a,t)=f(\sigma,a,t)&=0&\text{ on }\Sigma, \\
m(x,a,0)=m_0\quad f(x,a,0)=f_0&&\text{ in }Q_A,\\
m(x,0,t)=(1-\gamma)\int_{0}^{A}\beta(a,p)fda,\quad f(x,0,t)=\gamma \int_{0}^{A}\beta(a,p)fda&& \text{ in }Q_T.\\
\end{array}
\right.
\label{2}
\end{equation}
We have the following result.
\begin{proposition}
	Under the assumption $(H_1)-(H_2)-(H_3)-(H_4),$ for any fixed $(m_0,f_0)\in\left(L^{2}(Q_A)\right)^2\text{ and }(\chi_{\Xi}v,\chi_{\Xi'}u)\in \left(L^{2}(Q)\right)^2,$ and for any fixed $p\in L^2(Q_T)$  the system $(\ref{2})$ admits a unique solution \[(m,f)\in \left(L^2\left((0,A)\times(0,T);H^{1}_{0}(\Omega)\right)\right)^2\] 
 	and we have the following estimations:
 	\[\|m\|_{L^2\left((0,A)\times(0,T);H^{1}_{0}(\Omega)\right)}\leq K\left(\|f_0\|_{L^2(Q_A)}+\|m_0\|_{L^2(Q_A)}+\|\chi_{\Xi}v\|_{L^2(Q)}+\|v\|_{L^2(\Xi)}\right)\]
 	and 
 	\[\|f\|_{L^2\left((0,A)\times(0,T);H^{1}_{0}(\Omega)\right)}\leq C\left(\|f_0\|_{L^2(Q_A)}+\|u\|_{L^2(\Xi')}\right),\] where $K$ and $C$ are positive constants independent of $p.$
 \end{proposition}
\begin{proof}
It is obvious that the following system admits a unique solution according to the already existing results.
Indeed, as $\beta\in L^{\infty}((0,A)\times \mathbb{R})$ , then for every $p\in L^{2}(Q_T),$ the system:
 \begin{equation}
\left\{ 
\begin{array}{ccc}
f_t+f_a-K_f\Delta f+\mu_f f&=\chi_{\Xi'} u&\text{ in }Q,\\ 
f(\sigma,a,t)&=0&\text{ on }\Sigma, \\
f(x,a,0)=f_0&&\text{ in }Q_A,\\
\quad f(x,0,t)=\gamma \int_{0}^{A}\beta(a,p)fda&& \text{ in }Q_T\\
\end{array}
\right.
\label{22}
\end{equation}
admits a unique solution $f$ for all $f_0\in L^2(Q_A)$ and $\chi_{\Xi'}u\in L^2(Q)$ (see $\cite{ains}$).\\
We denote by $y=e^{-\lambda_0 t}f.$ The function $y$ verifies the following system
 \begin{equation}
\left\{ 
\begin{array}{ccc}
y_t+y_a-K_f\Delta y+(\lambda_0+\mu_f) y&=e^{-\lambda_0 t}\chi_{\Xi'} u &\text{ in }Q,\\ 
y(\sigma,a,t)&=0&\text{ on }\Sigma, \\
y(x,a,0)=f_0&&\text{ in }Q_A,\\
\quad y(x,0,t)=\gamma \int_{0}^{A}\beta(a,p)yda&& \text{ in }Q_T\\
\end{array}
\right.
\label{2o2}
\end{equation}
Multiplying the first equation of $(\ref{2o2})$ by $y$ and integrating the result over $Q,$ we obtain
\[
\dfrac{1}{2}\int_{0}^{A}\int_{\Omega}y^{2}(x,a,T)dxda+\dfrac{1}{2}\int_{0}^{T}\int_{\Omega}y_{\epsilon}^{2}(x,A,t)dxdt+\]\[K\int_{0}^{T}\int_{0}^{A}\int_{\Omega}\| \nabla y\|^{2}_{L^2(\Omega)}dxdadt+\int_{0}^{T}\int_{0}^{A}\int_{\Omega}(\mu_m+\lambda_0)y^{2} dxdadt\]
\begin{equation}
=\dfrac{1}{2}\|f_0\|^{2}_{L^2(\Omega\times (0,A))}+\gamma\int_{0}^{T}\int_{\Omega}\left(\int_{0}^{A}\beta(a,p)yda\right)^2 dxdt+\int_{0}^{T}\int_{0}^{A}\int_{\Omega}e^{-\lambda_0 t}\chi_{\Xi}uydxdadt
\end{equation}
Using Young inequality, Cauchy Schwarz inequality and the fact that $\beta$ is $L^\infty,$ we get that:
\[\gamma\int_{0}^{T}\int_{\Omega}\left(\int_{0}^{A}\beta(a,p)yda\right)^2dxdt+\int_{0}^{T}\int_{0}^{A}\int_{\Omega}\chi_{\Xi'}uydxdadt\]\[\leq A\|\beta\|^{2}_{\infty}\|y\|^{2}_{L^2(Q)}+\dfrac{1}{2}\|y\|^{2}_{L^2(Q)}+\dfrac{ 1}{2}\|u\|^{2}_{L^2(\Xi')}\]
Therefore, choosing $\gamma_0=(A\|\beta\|^{2}_{\infty}+3/2),$ we get:
\[
	\int_{0}^{A}\int_{\Omega}y^2(x,a,T)dxda+\int_{0}^{T}\int_{0}^{A}\int_{\Omega}(1+\mu_f )y^2dxdadt+K_f\int_{0}^{T}\int_{0}^{A}\int_{\Omega}|\nabla y|^2dxdadt\]
	\begin{equation}
		\leq 	\dfrac{1}{2}\|f_0\|^{2}_{L^2(\Omega\times (0,A))}+\dfrac{ 1}{2}\|u\|^{2}_{L^2(\Xi')}	.
		\end{equation}	
Then 
\[e^{-T(A\|\beta\|^{2}_{\infty}+3/2)}\left(\int_{0}^{T}\int_{0}^{A}\int_{\Omega}f^2dxdadt+K_f\int_{0}^{T}\int_{0}^{A}\int_{\Omega}|\nabla f|^2dxdadt\right)\] 
\[\leq 	\dfrac{1}{2}\|f_0\|^{2}_{L^2(\Omega\times (0,A))}+\dfrac{ 1}{2}\|u\|^{2}_{L^2(\Xi')}	.\]
Finally
\[ \|f\|_{L^2(0,A)\times (0,T);H^{1}_{0}(\Omega))}\leq \dfrac{1}{2(\min\{1,K_f\})e^{-T(A\|\beta\|^{2}_{\infty}+3/2)}}\left(\|f_0\|_{L^2(Q_A)}+\|u\|_{L^2(\Xi')}\right).\]
Likewise, as the boundary condition in age of the state $m$ is $(1-\gamma)\int_{0}^{A}\beta(a,p)fda$ then we have the existence and the uniqueness of $m$ and there exists $K>0$ independent of $p$ such that $m$ verify: 
\[ \|m\|_{L^2\left((0,A)\times(0,A);H^{1}_{0}(\Omega)\right)}\leq K\left(\|f_0\|_{L^2(Q_A)}+\|m_0\|_{L^2(Q_A)}+\|v\|_{L^2(\Xi)}+\|u\|_{L^2(\Xi')}\right).\]
\end{proof}
Next, we will establish an observability inequality.
\subsection{Observability inequality}
The adjoint system of the auxiliary system $(\ref{2})$ is given by:
\begin{equation}
\left\{ 
\begin{array}{ccc}
-n_t-n_a-K_m\Delta n+\mu_m n&=0 &\text{ in }Q ,\\ 
-l_t-l_a-K_f\Delta l+\mu_f l&=(1-\gamma)\beta(a,p)n(x,0,t)+\gamma\beta(a,p)l(x,0,t)&\text{ in }Q,\\ 
n(\sigma,a,t)=l(\sigma,a,t)&=0&\text{ on }\Sigma, \\
n(x,a,T)=n_T\quad l(x,a,T)=l_T&&\text{ in }Q_A,\\
n(x,A,t)=0,\quad l(x,A,t)=0&& \text{ in }Q_T.\\
\end{array}
\right.
\label{3}
\end{equation}
The main idea in this part is to establish an observability inequality of the adjoint system that will allow us to prove the approximate null controllability of the system $(\ref{2})$.	\\The basic idea for establishing this inequality is the estimation of he non-local terms. To this end, let us start by formulating a representation of the solution of the adjoint cascade system by using semi-group theory and characteristic's method.\\For every $(n_T,l_T)\in (L^2(Q_A))^2, $ under the assumptions $(H_1)$ and $(H_2),$ the coupled system $(\ref{3})$ admits a unique solution $(n,l).$ Moreover integrating along the characteristic lines, the solution $(n,l)$ of $(\ref{3})$ is given by:
		\begin{equation}
		n(t) = 
		\begin{dcases}
			\dfrac{\pi_1(a+T-t)}{\pi_1(a)}e^{(T-t) K_m\Delta} n_T(x,a+T-t)\text{ if } T-t\leq A-a, \\
			0 \text{ if } A-a<T-t \label{alp}
		\end{dcases}
		\end{equation}
		and
		\begin{equation}
		l(t) = 
		\begin{dcases}
			\dfrac{\pi_2(a+T-t)}{\pi_2(a)}e^{(T-t) K_f\Delta}l_T(x,a+t-T)\\
			+\int_{t}^{T}\dfrac{\pi_2(a+s-t)}{\pi_2(a)}e^{(s-t) K_f\Delta}\beta(a+s-t,p(x,s))\left((1-\gamma)n(x,0,s)+\gamma l(x,0,s)\right)ds\text{ if } T-t\leq A-a, \\
			\int_{t}^{t+A-a}\dfrac{\pi_2(a+s-t)}{\pi_2(a)}e^{(s-t) K_f\Delta}\beta(a+s-t,p(x,s))\left((1-\gamma)n(x,0,s)+\gamma l(x,0,s)\right)ds \text{ if } A-a<T-t, \label{alp}
		\end{dcases}
				\end{equation}
				where $\pi_1(a)=e^{-\int_{0}^{a}\mu_m(r)dr}\text { , }\pi_2(a)=e^{-\int_{0}^{a}\mu_f(r)dr}\text{ and } e^{tK_m \Delta}$ is the semi-group of $-K_m\Delta$ with Dirichlet boundary condition. 
				Under the assumptions $(H_1)-(H_2)-(H_3)$ and $(H_4)$ the couple $(n,l)$ of the system $(\ref{3})$ verifies the following result.
\begin{theorem}
			Under the assumptions of the Theorem 1.1, there exists a constant $C_{T}>0$ independent of $p$ such that the solution $(n,l)$ of the system $(\ref{3})$ verifies:
					\begin{equation}	
					\int_{0}^{A}\int_{\Omega}l^2(x,a,0)dxda+\int_{0}^{A}\int_{\Omega}n^2(x,a,0)dxda\leq C_{T}\left(\int_{\Xi}n^2(x,a,t)dxdadt+\int_{\Xi'}l^2(x,a,t)dxdadt\right).\label{12}\end{equation}
		\end{theorem}
		For proof of the Theorem 2.1 we state the following approximation of the non-local terms.
		\subsubsection{Estimations of the non-local terms}
		\begin{proposition}
			Under the assumptions of the Theorem 1.1, for every $\eta$ such that $a_1<\eta<T,$ there exists a positive constant $C$ such that the following inequality:
			\begin{equation}
\int_{0}^{T-\eta}\int_{\Omega}n^2(x,0,t)dxdt\leq C\int_{0}^{T}\int_{a_1}^{a_2}\int_{\omega}n^2(x,a,t)dxdadt\label{d0d}
\end{equation}
holds.\\
In particular, for every $\varrho>0,$ if $a_1=0$ and $n_T=0\text{ a.e in } \Omega\times(0,\varrho),$ there is a constant $C_{\varrho,T}$ such that:
\begin{equation}
\int_{0}^{T}\int_{\Omega}n^2(x,0,t)dxdt\leq C_{\varrho,T}\int_{0}^{T}\int_{0}^{a_2}\int_{\omega}n^2(x,a,t)dxadt.\label{14}
\end{equation}
Moreover, if $b_1<b,$ for every $\eta$ such that $b_1<\eta<T,$ there exists a positive constant such that the following inequality
\begin{equation}
\int_{0}^{T-\eta}\int_{\Omega}l^2(x,0,t)dxdt\leq C\int_{\Xi'}l^2(x,a,t)dxdadt \label{dd}
\end{equation}
holds.
\end{proposition}
\begin{remark}
In fact, as one can see the assumption $(H_3)-(i)$ is not useful for the proof of the inequality $(\ref{d0d})$.
	\end{remark}
In order to prove the Proposition 2.2, we first recall the following observability inequality for parabolic equation (see, for instance $\cite{b10}$):
\begin{proposition}
	Let $T>0,$ $t_0$ and $t_1$ be such that $0<t_0<t_1<T.$ Then for every $w_0\in L^2(\Omega),$ the solution $w$ of the following system \
	\begin{equation}
		\left\{
		\begin{array}{ccc}
		\dfrac{\partial w(x,\lambda)}{\partial \lambda}-\Delta w(x,\lambda)=0\text{ in } (t_0,T)\times \Omega,\\
		 w=0 \text{ on } (t_0,T)\times \partial\Omega,	\\
		w(x,t_0)=w_0(x)\text{ in }\Omega,
		\end{array}
		\right.
	\end{equation}
	satisfies the estimate
	\[\int_{\Omega}w^2(T,x)dx\leq\int_{\Omega}w^2(x,t_1)dx\leq c_1e^{\dfrac{c_2}{t_1-t_0}}\int_{t_0}^{t_1}\int_{\omega}w^2(x,\lambda)dxd\lambda,\]
	where the constant $c_1$ and $c_2$ depend on $T$ and $\Omega.$
\end{proposition}	
\begin{proof}{of the Proposition 2.2}\\
The state $n$ of the system $(\ref{3})$ verifies:
\begin{equation}
\left\{ 
\begin{array}{ccc}-\dfrac{\partial n}{\partial t}-\dfrac{\partial n}{\partial a}-\Delta n+\mu_m n&=0&\text{ in }\Omega\times(0,a_2)\times (0,T),\\
n&=0& \text{ on }\partial\Omega\times (0,A)\times (0,T),\\
n(a,T)&=n_T& \text{ in }\Omega\times(0,a_2).
\end{array}
\right.
\label{ad}
\end{equation}	
Let $\tilde{n}(x,a,t)=n(x,a,t)e^{-\int_{0}^{a}\mu(\alpha)d\alpha}.$ Then $\tilde{n}$ satisfies
\begin{equation}
\left\{
\begin{array}{ccc}
	\dfrac{\partial \tilde{n}}{\partial t}+\dfrac{\partial \tilde{n}}{\partial a} +\Delta\tilde{n}&=0&\text{ in }\Omega\times (a_1,a_2)\times (0,T), \\
\hat{n}&=0&\text{ on } \partial \Omega\times (0,A)\times (0,T),\\
\hat{n}(.,.,T)&=n_Te^{-\int_{0}^{a}\mu_m(\alpha)d\alpha}&\text{ in }\Omega\times (0,A).
\end{array}
\right.\label{ad11}
\end{equation}
Proving the inequality $(\ref{d0d})$ leads also to show that,
there exits a constant $C>0$ such that the solution $\tilde{n}$ of $(\ref{ad11})$ satisfies 
\begin{equation}
\int_{0}^{T-\eta}\int_{\Omega}\tilde{n}^2(x,0,t)dxdt\leq C\int_{0}^{T}\int_{a_1}^{a_2}\int_{\omega}\tilde{n}(x,a,t)dxdadt. \label{cci}
\end{equation}
Indeed, we have 
\[
\int_{0}^{T-\eta}\int_{\Omega}n^2(x,0,t)dxdt=\int_{0}^{T-\eta}\int_{\Omega}\tilde{n}^2(x,0,t)dxdt\]\[
\leq C\int_{0}^{T}\int_{a_1}^{a_2}\int_{\omega}\tilde{n}^2(x,a,t)dxdadt=C\int_{0}^{T}\int_{a_1}^{a_2}\int_{\omega}e^{-2\int_{0}^{a}\mu_m(r)dr}n^2(x,a,t)dxdadt\]\[\leq C\int_{0}^{T}\int_{a_1}^{a_2}\int_{\omega}n^2(x,a,t)dxdadt.
\]
We consider the following characteristics trajectory $\gamma(\lambda)=(T-\lambda,T+t-\lambda).$ If $T-\lambda=0$ the backward characteristics starts from $(0,t).$ If $T<a_1$ the trajectory $\gamma(\lambda)$ never reaches the observation region $(a_1,a_2)$ (see Fig 1). So we choose $T>a_1.$\\
	Without loss of the generality, let us assume here $\eta< a_2<T.$\\
	The proof is now done in two steps:
	\[\textbf{step 1: Estimation for } t\in (0,T-a_2)\]
We put:\\
\[w(\lambda)=\tilde{n}(x,T-\lambda,T+t-\lambda) \text{ ; }(\lambda\in (T-a_2,T)\quad and \quad x\in \Omega).\]
	Then, $w$ satisfies:
	\begin{equation}
		\left\{
		\begin{array}{ccc}
		\dfrac{\partial w(\lambda)}{\partial \lambda}-\Delta w(\lambda)=0\text{ in } \Omega\times (T-a_2,T)'\\
		w=0\text{ on } \partial \Omega\times (T-a_2,T),\\
		w(0)	=\tilde{n}(x,T,T+t)	\text{ in }\Omega.	
		\end{array}
		\right.
	\end{equation}
	Using the Proposition 2.3  with $T-a_2<t_0<t_1<T$ we obtain:\\
\[\int_{\Omega}w^2(T)dx\leq\int_{\Omega}w^2(t_1)dx\leq c_1e^{\dfrac{c_2}{t_1-t_0}}\int_{t_0}^{t_1}\int_{\Omega}w^2(\lambda)dxd\lambda.\]
That is equivalent to
\[\int_{\Omega}\tilde{n}^2(x,0,t)dx\leq c_1e^{\dfrac{c_2}{t_1-t_0}}\int_{t_0}^{t_1}\int_{\Omega}\tilde{n}^2(x,T-\lambda,t+T-\lambda)dxd\lambda\]
\[C\int_{T-t_1}^{T-t_0}\int_{\Omega}\tilde{n}^2(x,a,t+a)dxda.\]	
Then, for $t_0=T-a_2$ and $t_1=T-a_1,$ we obtain\\
\[\int_{\Omega}\tilde{n}^2(x,0,t)dxdt\leq C\int_{a_1}^{a_2}\int_{\omega}\tilde{n}^2(x,a,t+a)dxda.\]
Integrating with respect to $t$ over $(0,T-a_2)$ we get
\[\int_{0}^{T-a_2}\int_{\Omega}\tilde{n}^2(x,0,t)dxdt\leq C
\int_{a_1}^{a_2}\int_{a}^{T-a_2+a}\int_{\omega}\tilde{n}^2(x,a,t)dxdt da.\]
Finally, we have
\begin{equation}
	\int_{0}^{T-a_2}\int_{\Omega}\tilde{n}^2(x,0,t)dxdt\leq C
\int_{0}^{T}\int_{a_1}^{a_2}\int_{\omega}\tilde{n}^2(x,a,t)dxdadt.\label{z}
\end{equation}
\[\textbf{Step 2: Estimation for } t\in (T-a_2,T-\eta)\quad where \quad \eta\in (a_1,a_2)\]
Setting:\\
\[w(\lambda)=\tilde{n}(x,T-\lambda,T+t-\lambda) \text{ ; }(\lambda\in (T-a_1,T)),\] $w$ satisfies:
	\begin{equation}
		\left\{
		\begin{array}{ccc}
		\dfrac{\partial w(\lambda)}{\partial \lambda}-\Delta w(\lambda,x)=0\text{ in }  \Omega\times (T-a_2,T),\\
		w=0\text{ on } \partial \Omega\times (T-a_2,T),\\
				w(T-a_2)	=\tilde{n}(x,a_2,t+a_2)\text{ in }\Omega.
		\end{array}
		\right.
	\end{equation}
	Using again the Proposition 2.3  with $T-a_2<t_0<t_1<T$ we obtain:\\
\[\int_{\Omega}w^2(T)dx\leq\int_{\Omega}w^2(t_1)dx\leq c_1e^{\dfrac{c_2}{t_1-t_0}}\int_{t_0}^{t_1}\int_{\omega}w^2(\lambda)dxd\lambda.\]
Therefore,
\[\int_{\Omega}\tilde{n}^2(x,0,t)dx\leq c_1e^{\dfrac{c_2}{t_1-t_0}}\int_{t_0}^{t_1}\int_{\omega}\tilde{n}(x,T-\lambda,T+t-\lambda)dxd\lambda.\]	
Then, for $t_0=T-\eta$ and $t_1=T-a_1,$ we obtain\\
\[\int_{\Omega}\tilde{n}^2(x,0,t)dx\leq c_1e^{\dfrac{c_2}{\eta-a_1}}\int_{a_1}^{\eta}\int_{\omega}\tilde{n}(x,a,t+a)dxda.\]
Integrating with respect to $t$ over $(T-a_2,T-\eta)$ we get,
\[\int_{T-a_2}^{T-\eta}\int_{\Omega}\tilde{n}^2(x,0,t)dxdt\leq c_1e^{\dfrac{c_2}{\eta-a_1}}\int_{T-a_2}^{T-\eta}\int_{a_1}^{\eta}\int_{\omega}\tilde{n}^2(x,\alpha,t+\alpha)dxd\alpha dt\]\[\leq c_1e^{\dfrac{c_2}{\eta-a_1}}\int_{a_1}^{\eta}\int_{T-a_2}^{T-\eta}\int_{\omega}\tilde{n}^2(x,\alpha,t+\alpha)dxdtd\alpha .\]
Finally, one gets
\begin{equation}
	\int_{T-a_2}^{T-\eta}\int_{\Omega}\tilde{n}^2(x,0,t)dxdt\leq c_1e^{\dfrac{c_2}{\eta-a_1}}
\int_{0}^{T}\int_{a_1}^{a_2}\int_{\omega}\tilde{n}^2(x,a,t)dxdadt.\label{zz}
\end{equation}
Combining $(\ref{z})$ and $(\ref{zz}),$ we obtain:
\begin{equation}
\int_{0}^{T-\eta}\int_{\Omega}n^2(x,0,t)dxdt\leq C(\eta,a_1,a_2,\Omega)\int_{0}^{T}\int_{a_1}^{a_2}\int_{\omega}n^2(x,a,t)dxdadt.\label{ou1}
\end{equation}
Notice that, as $\lim_{\eta\rightarrow a_1}c_1e^{\dfrac{c_2}{\eta-a_1}}=+\infty$ then $\lim_{\eta\rightarrow +a_1} C(\eta,a_1,a_2,\Omega)=+\infty.$\\
Suppose now that $a_1=0$. From the above, we have for all $ \varrho>0,$ the existence of a constant depending on $ \varrho $ such that:
\begin{equation}
\int_{0}^{T-\varrho}\int_{\Omega}n^2(x,0,t)dxdt\leq C(T,\varrho,a_2,\Omega)\int_{0}^{T}\int_{0}^{a_2}\int_{\omega}n^2dxdadt.\label{ou2}
\end{equation}
Moreover $n_T=0\text{ in } \Omega\times (0,\varrho),$ then by using the characteristics method, we obtain $n(x,0,t)=0 \text{ in } \Omega\times (T-\varrho,T)$ (see Figure 1 (b)).\\
Finally, if $n_T=0\text{ in } \Omega\times (0,\varrho),$ we have the following estimates:
\begin{equation}
\int_{0}^{T}\int_{\Omega}n^2(x,0,t)dxdt\leq C(T,\varrho,a_2,\Omega)\int_{0}^{T}\int_{0}^{a_2}\int_{\omega}n^2dxdadt.\label{ou3}
\end{equation}
Taking $X=\min\{b,b_2\},$ the adjoint system relating to the state $l$ can be rewritten as:
\begin{equation}
\left\{ 
\begin{array}{ccc}-\dfrac{\partial l}{\partial t}-\dfrac{\partial l}{\partial a}-\Delta l+\mu_f l&=0&\text{ in }\Omega\times(0,X)\times (0,T),\\ 
l&=0& \text{ on }\partial\Omega \times (0,X)\times (0,T),\\
l(a,T)&=l_T& \text{ in }\Omega\times(0,X).
\end{array}
\right.
\label{mnk}
\end{equation}	
Proceeding as above, we get the inequality $(\ref{dd})$
\end{proof}
\begin{figure}[H]
\begin{subfigure}{0.50\textwidth}	
			\begin{overpic}[scale=0.22]{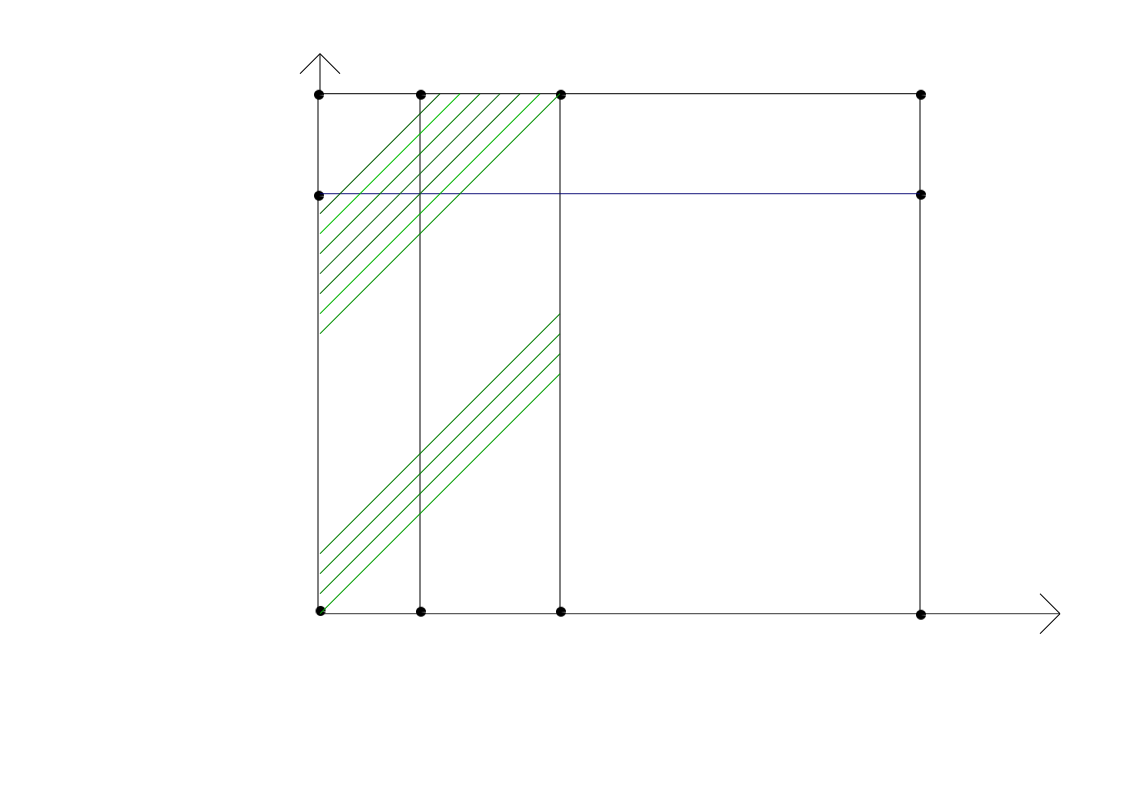}
			\put (22,62) {$T$}
			\put (14,52.5) {$T-a_1$}
			\put (35,12) {$a_1$}
			\put (49,12) {$a_2$}
			\put (80,10.5) {$A$}
              \end{overpic}
		\caption{An illustration of the estimate of $n(x,0,t)\text{ and } l(x,0,t).$ Here we have chosen $a_1=b_1$ and $a_2=b_2=b.$ Since $t\in (0,T-a_1)$ all the backward characteristics starting from $(0,t)$ enters the observation domain.}
		\end{subfigure}\quad	
		\begin{subfigure}{0.50\textwidth}
			\begin{overpic}[scale=0.25]{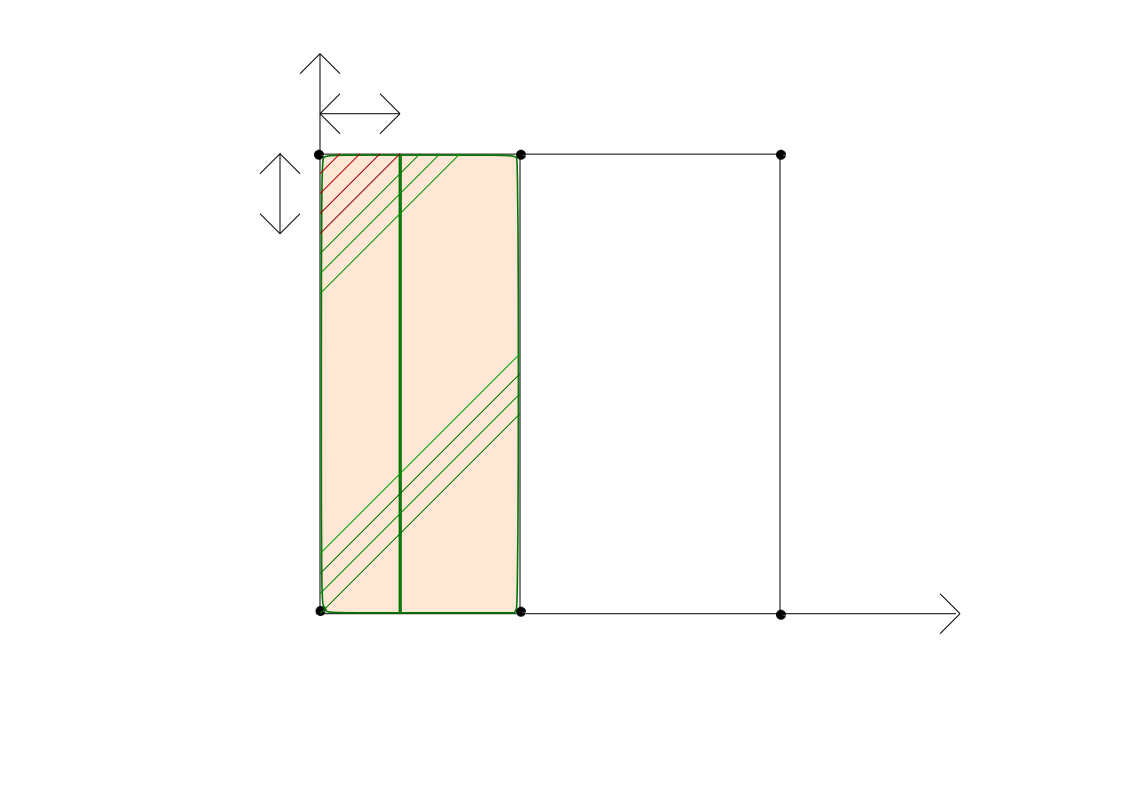}
			\put (29,62) {$n_T=0$}
			\put (4,52.5) {$n(x,0,t)=0$}
			\put (35,13) {$\varrho$}
			\put (45.5,13) {$a_2$}
			\put (68,11.5) {$A$}
			 \end{overpic}
		\caption{An illustration of the estimate of
$\int_{0}^{T}\int_{\Omega}n^2(x,0,t)dxdt. \text{ Here }n_T(x,a)=0\text{ in } \Omega\times(0,\varrho)$ and then $n(x,0,t)=0\text{ for } t\in (T-\varrho,T)$.}
		\end{subfigure}
			\caption{Illustration of the estimation of non-local terms}		\end{figure}
			\subsubsection{Estimation of $l(x,a,0)$ and $n(x,a,0)$}
We also need the following estimates
\begin{proposition}
Under assumptions $(H_1)-(H3),$ for every $T>\sup\{a_1,A-a_2\}$, there exists $C_T>0$ such that the solution $(n,l)$ of the system $(\ref{3})$ verifies the following inequality:
\begin{equation}
	\int_{0}^{A}\int_{\Omega}n^{2}(x,a,0)dxda\leq C_T\int_{\Xi}n^{2}(x,a,t)dxdadt. \label{10}
\end{equation}
\end{proposition}
For every $T>\sup\{a_1,A-a_2\},$ there exists $a_0\in (a_1,a_2)$ such that $n(x,a,0)=0$ for all $(x,a)\in \Omega\times (a_0,A).$
The result follows from the stated below lemma.
\begin{lemma}
		Let us suppose that $T>\sup\{a_1,A-a_2\}$ . Then, there exists $a_0 \in (a_1,a_2)$ such that
		\[T>A-a_0>A-a\text{ for all } a\in (a_0,A),\] therefore, \[n(x,a,0)=0\text{ for all }x\in \Omega\text{ and } a\in (a_0,A).\]
		\end{lemma}
		\begin{proof}{of the Lemma 2.1}\\
			Suppose that $T>A-a_2,$ then, there exists $\kappa>0$ (we choose $\kappa$ such that $\kappa<a_2-a_1$ it reassures that $a_1<a_2-\kappa$) such that $T>A-a_2+\kappa\Leftrightarrow T>A-(a_2-\kappa).$\\
			We denote by $a_0=a_2-\kappa.$ And as $A-a_0>A-a\quad for \quad all\quad a\in (a_0,A),$ then, $T-0>A-a_0\quad for\quad all\quad a\in (a_0,A).$
			Finally, as $n(x,a,t)=0$ for almost all $(a,t)$ such that $T-t> A-a,$ we get
			$n(x,a,0)=0\text{ in }(x,a)\in \Omega\times (a_0,A)$.
 			\end{proof}
			From the Lemma 2.1, we therefore have to prove the following inequality:	
			\begin{equation}
	\int_{0}^{a_0}\int_{\Omega}n^{2}(x,a,0)dxda\leq C_T\int_{0}^{T}\int_{a_1}^{a_2}\int_{\omega}n^{2}(x,a,t)dxdadt \label{12v2}
\end{equation}
\begin{proof}{of the Proposition 2.4}\\
We consider the state $n$ of the cascade system $(\ref{3})$ that verifies the following system:
\begin{equation}
\left\{ 
\begin{array}{ccc}-\dfrac{\partial n}{\partial t}-\dfrac{\partial n}{\partial a}-\Delta n+\mu_m n&=0&\text{ in }\Omega\times(0,A)\times (0,T),\\ 
n&=0& \text{ on }\partial\Omega \times (0,A)\times (0,T),\\
n(x,a,T)&=n_T& \text{ in }\Omega\times (0,A).
\end{array}
\right.\label{rrv}
\end{equation}	
We denote by $\tilde{n}(x,a,t)=n(x,a,t)e^{-\int_{0}^{a}\mu_m(\alpha)d\alpha}.$ Then, $\tilde{n}$ satisfies
\begin{equation}
\left\{ 
\begin{array}{ccc}
\dfrac{\partial \tilde{n}}{\partial t}+\dfrac{\partial \tilde{n}}{\partial a} +\Delta\tilde{n}&=0&\text{ in }\Omega\times (0,A)\times (0,T),\\ 
\tilde{n}&=0& \text{ on }\partial\Omega \times (0,A)\times (0,T),\\
\tilde{n}(x,a,T)&=n_Te^{-\int_{0}^{a}\mu_m(r)dr}& \text{ in }\Omega\times (0,A).
\end{array}
\right.
\label{ad}
\end{equation}
Proving the inequality $(\ref{12v2})$ leads also to show that,
there exits a constant $C>0$ such that the solution $\tilde{n}$ of $(\ref{rrv})$ satisfies 
\begin{equation}
\int_{0}^{a_0}\int_{\Omega}\tilde{n}^2(x,a,0)dxda\leq C\int_{0}^{T}\int_{a_1}^{a_2}\int_{\omega}\tilde{n}(x,a,t)dxdadt. \label{cci}
\end{equation}
Indeed, we have 
\[
\int_{0}^{a_0}\int_{\Omega}n^2(x,a,0)dxda=\int_{0}^{a_0}\int_{\Omega}e^{2\int_{0}^{a}\mu_m(r)dr}\tilde{n}^2(x,a,0)dxda\leq Ce^{2\int_{0}^{a_0}\mu_m(\alpha)d\alpha}\int_{0}^{T}\int_{a_1}^{a_2}\int_{\omega}\tilde{n}^2(x,a,t)dxdadt.\]
Then
\[\int_{0}^{a_0}\int_{\Omega}n^2(x,a,0)dxda\leq Ce^{2\|\mu_m\|_{L^1(0,a_0)}}\int_{0}^{a_0}\int_{\Omega}\tilde{n}^2(x,a,0)dxda\leq Ke^{2\|\mu_m\|_{L^1(0,a_0)}}\int_{0}^{T}\int_{a_1}^{a_2}\int_{\omega}n^2(x,a,t)dxdadt.
\]
			We consider in this proof the characteristics $\gamma (\lambda)=(a+\lambda,\lambda).$
		For $\lambda=0$ the characteristics starts from $(a,0).$\\
		We have two cases.\\
		\textbf{Case 1}: $T<a_2$ \\
	Two situations can arise: \\
	$\bullet$ $b_0=a_2-T<a_1<a_0$ in this situation we split the interval $(0,a_0)$ as \begin{equation}(0,a_0)=(0,b_0)\cup(b_0,a_1)\cup (a_1,a_0)\label{iu}.
	\end{equation}\\
	$\bullet$ $a_1<b_0=a_2-T<a_0$, here we split the interval $(0,a_0)$ as
	\[(0,a_0)=(0,a_1)\cup (a_1,a_0).\]	
			\textbf{Case 2}: $T\geq a_2$ \\ In this case we split the interval $(0,a_0)$ as $(0,a_0)=(0,a_1)\cup (a_1,a_0).$\\
	In the remaining part of the proof we give upper bounds for $\int_{I}\int_{\Omega}\tilde{n}^2(x,a,0)dxda$
 where $I$ is successively each one of the intervals appearing in the decomposition $(\ref{iu})$.\\
	\textbf{Upper bound on $(0,b_0)$}:\\
	For $a\in(0,b_0)$ we first set 
	$w(x,\lambda)=\tilde{n}(x,T+a-\lambda,T-\lambda)\text{    } (\lambda\in (0,T)\text{ and } x\in\Omega)$ where $\tilde{n}=e^{-\int_{0}^{a}\mu_m(\alpha)d\alpha}n.$\\
	Then, $w$ verifies 
	\begin{equation}
		\left\{
		\begin{array}{ccc}
		\dfrac{\partial w(x,\lambda)}{\partial \lambda}-\Delta w(x,\lambda)=0\text{ in } \Omega\times(0,T),\\
		w=0 \text{ on } \partial \Omega \times (0,T),\\
		w(T)=\tilde{n}(x,a+T,T)\text{ in }\Omega.
		\end{array}
		\right.\label{ooo}
	\end{equation}
	By applying the Proposition 2.3 with $t_0=0$ and $t_1=a+T-a_1,$ we obtain:\\
\[\int_{\Omega}w^2(x,T)dx\leq c_1e^{\dfrac{c_2}{a+T-a_1}}\int_{0}^{a+T-a_1}\int_{\omega}w^2(x,\lambda)dxd\lambda.\]
	
	Then, we have 
\[\int_{\Omega}\tilde{n}^2(x,a,0)dx\leq c_1e^{\dfrac{c_2}{a+T-a_1}}\int_{a_1}^{a+T}\int_{\omega}\tilde{n}(x,\alpha,\alpha-a)dxd\lambda\]\[=C\int_{a_1}^{a+T}\int_{\omega}\tilde{n}^2(x,\alpha,\alpha-a)dxd\alpha\leq C\int_{a_1}^{a_2}\int_{\omega}\tilde{n}(x,\alpha,\alpha-a)dxd\alpha.\]	
Integrating with respect to $a$ over $(0,b_0)$ we get
\[\int_{0}^{b_0}\int_{\Omega}\tilde{n}^2(x,a,0)dxda\leq C\int_{0}^{b_0}\int_{a_1}^{a_2}\int_{\omega}\tilde{n}^2(x,\alpha,\alpha-a)dxd\alpha da.\]
As
\[\int_{0}^{b_0}\int_{a_1}^{a_2}\int_{\omega}\tilde{n}^2(x,\alpha,\alpha-a)dxd\alpha da=\int_{a_1}^{a_2}\int_{0}^{b_0}\int_{\omega}\tilde{n}^2(x,\alpha,\alpha-a)dxdad\alpha ,\]
then,
\[\int_{0}^{b_0}\int_{\Omega}\tilde{n}^2(x,a,0)dxda\leq C\int_{a_1}^{a_2}\int_{\alpha-b_0}^{\alpha}\int_{\omega}\tilde{n}^2(x,\alpha,t)dxdtd\alpha .\]
Finally
\begin{equation}
	\int_{0}^{b_0}\int_{\Omega}\tilde{n}^2(x,a,0)dxda\leq C\int_{0}^{T} \int_{a_1}^{a_2}\int_{\omega}\tilde{n}^2(x,a,t)dxdadt. \label{7}
\end{equation}
\textbf{Upper bound on $(b_0,a_1)$}:\\
For $a\in (b_0,a_1)$	we consider always the system $(\ref{ooo})$ but $\lambda\in (T+a-a_2,T).$\\
Applying the Proposition 2.3 with $t_0=a+T-a_2$ and $t_1=a+T-a_1$, we obtain
\[\int_{\Omega}\tilde{n}^2(x,a,0)dx\leq C\int_{a_1}^{a_2}\int_{\omega}\tilde{n}^2(x,\alpha,\alpha-a)dxd\alpha.\]	
And as before, we get 
\begin{equation}
	\int_{b_0}^{a_1}\int_{\Omega}\tilde{n}^2(x,a,0)dxda\leq C\int_{0}^{T}\int_{a_1}^{a_2}\int_{\omega}\tilde{n}^2(x,a,t)dxdadt. \label{9}
	\end{equation}
	$\textbf{Upper bound on } (a_1,a_0)$\\
	Similarly, we obtain for $t_0=T+a-a_2$ and $t_1=T$:
		\begin{equation}
	\int_{a_1}^{a_0}\int_{\Omega}\tilde{n}(x,a,0)dxda\leq C\int_{0}^{T}\int_{a_1}^{a_1}\int_{\omega}\tilde{n}^2(x,a,t)dxdadt. \label{1555}
	\end{equation}
	Consequently, combining $(\ref{7})$, $(\ref{9})$ and $(\ref{1555})$ we obtain:
	\[\int_{0}^{a_0}\int_{\Omega}\tilde{n}^2(x,a,0)dxda\leq C\int_{0}^{T}\int_{a_1}^{a_2}\int_{\omega}\tilde{n}(x,a,t)dxdadt.\]	
		\end{proof} 
We also need the following estimate.
\begin{proposition}
Let us assume  true the assumption $(H_1)-(H_2)$ and let $b_1<a_0<b$ and $T>b_1.$ There exists $C_T>0$ such that the solution $l$ of the system $(\ref{3})$ verifies the following inequality
\begin{equation}
	\int_{0}^{a_0}\int_{\Omega}l^{2}(x,a,0)dxda\leq C_T\int_{\Xi'}l^{2}(x,a,t)dxdadt. \label{C}
\end{equation}
\end{proposition}
\begin{proof}
From the assumption $(H_2),$ we have $\beta=0\text{ in }(0,b).$ 
Then the function $l$ verifies:
	\begin{equation}
\left\{ 
\begin{array}{ccc}-\dfrac{\partial l}{\partial t}-\dfrac{\partial l}{\partial a}-\Delta l+\mu_{l} l&=0&\text{ in }\Omega\times(0,b)\times (0,T),\\ 
l&=0& \text{ on }\partial\Omega \times (0,b)\times (0,T),\\
l(x,a,T)&=l_T& \text{ in }\Omega\times(0,b).
\end{array}
\right.
\label{a11}
\end{equation} 
Proceeding as in the Proof of Proposition 2.4, we get the desired result.
\end{proof}
\subsubsection{Proof of the observability inequality}
For the proof of the Theorem 2.1, we start with the following Lemma:
\begin{lemma}
		Let us suppose that $T>A-a_2+a_1$ and $a_1<b$. Then there exists $a_0 \in (a_1,b)$ and $\kappa>0$ such that
		\[T>T-(a_1+\kappa)>A-a_0>A-a\text{ for }\text{ all } a\in (a_0,A).\]
		Therefore \[l(x,a,0)=\int_{0}^{A-a}\dfrac{\pi(a+s)}{\pi(a)}\left(e^{s\Delta}\beta(a+s,p(x,s))l(x,0,s)+e^{s\Delta}\beta(a+s,p(x,s))n(x,0,s)\right)ds\quad for \text{ all } (x,a)\in\Omega\times(a_0,A).\]
		\end{lemma}
		\begin{proof}
		Notice that here the solution of the adjoint system $(\ref{3})$ is
		\begin{equation}
		n(t) = 
		\begin{dcases}
			 \dfrac{\pi_1(a+T-t)}{\pi_1(a)}e^{(T-t) K_m\Delta}n_T(x,a+T-t)\text{ if } T-t\leq A-a, \\
			0 \text{ if } A-a<T-t .\label{alp}
		\end{dcases}
		\end{equation}
		and
		\begin{equation}
		l(t) = 
		\begin{dcases}
			 \dfrac{\pi_2(a+T-t)}{\pi_2(a)}e^{(T-t) K_f\Delta}l_T(x,a+t-T)\\
			 +\int_{t}^{T}\dfrac{\pi_2(a+s-t)}{\pi_2(a)}\left(e^{(s-t)K_f\Delta}\beta(a+s-t,p(x,s))((1-\gamma)n(x,0,s)+\gamma l(x,0,s)\right)ds\text { if } T-t\leq A-a, \\
			\int_{t}^{t+A-a}\dfrac{\pi_2(a+s-t)}{\pi_2(a)}\left(e^{(s-t)k_f\Delta}\beta(a+s-t,p(x,s))\left((1-\gamma)n(x,0,s)+\gamma l(x,0,s)\right)\right)ds \text{ if }A-a<T-t .\label{alp}
		\end{dcases}
		\end{equation}		
		Without loss of the generality, we suppose that $a_2=b.$\\
			Suppose that $T>a_1+A-a_2$ then $T-a_1>A-a_2.$ So, there exists $\kappa>0$ (we choose $\kappa$ such that $2\kappa<a_2-a_1$ it reassures that $a_1+\kappa<a_2-\kappa$) such that $T-a_1-2\kappa>A-a_2$ that is $T-(a_1+\kappa)>A-(a_2-\kappa).$\\
			We denote by $a_0=a_2-\kappa$ and as $A-a_0>A-a\quad for \quad all\quad a\in (a_0,A),$ then $T>T-(a_1+\kappa)>A-a_0\quad for\quad all\quad a\in (a_0,A).$\\
			Moreover, for $(a,t)$ such that $T-t>A-a,$ we have \[l(x,a,t)=\int_{t}^{t+A-a}\dfrac{\pi_2(a+s-t)}{\pi_2(a)}\left(e^{K_f(s-t)\Delta}\beta(a+s-t,p(x,s))n(x,0,s)+e^{K_f(s-t)\Delta}\beta(a+s-t,p(x,s))l(x,0,s)\right)ds\] and as for $t=0$ and $a\in (a_0,A),\quad T-0>A-a_0>A-a,$ then \[l(x,a,0)=\int_{0}^{A-a}\dfrac{\pi_2(a+s)}{\pi_2(a)}\left(e^{(s K_f\Delta}\beta(a+s,p(x,s))n(x,0,s)+e^{sK_f\Delta}\beta(a+s,p(x,s))l(x,0,s)\right)ds.\]
			Notice that, as $(a_1,a_2)\subset (b_1,b_2),$ then, if $a_0\in (a_1,a_2),$ we have $a_0\in (b_1,b_2).$		
			\end{proof}
			Now, we can prove the result of Theorem 2.1
			\begin{proof}{of the Theorem 2.1}\\
				Let $a_0$ as in the Lemma 2.2. We have:
				\[\int_{0}^{A}\int_{\Omega}l^2(x,a,0)dxda=\int_{0}^{a_0}\int_{\Omega}l^2(x,a,0)dxda+\int_{a_0}^{A}\int_{\Omega}l^2(x,a,0)dxda.\]
				Using the results of the Lemma 2.2, the fact that $\beta\in L^{\infty}((0,A)\times \mathbb{R})$ and the fact that  
				$$\left\|\dfrac{\pi_2(a+t)}{\pi_2(a)}e^{t K_f\Delta}\right\|\leq 1,$$ we can prove the existence of a constant $K_T(A,T,\|\beta\|_{\infty})$ independent of $p$ such that:
				\[\int_{a_0}^{A}\int_{\Omega}l^2(x,a,0)dxda\leq K_T\left(\int_{0}^{T-(a_1+\kappa)}\int_{\Omega}n^2(x,0,t)dxdt+\int_{0}^{T-(a_1+\kappa)}\int_{\Omega}l^2(x,0,t)dxdt\right).\]
				Moreover, we have $b_1\leq a_1\leq a_1+\kappa,$ then, we have also, from the Proposition 2.1 that
				\[\int_{0}^{T-(a_1+\kappa)}\int_{\Omega}n^2(x,0,t)dxdt\leq C_T\int_{\Xi}n^{2}(x,a,t)dxdadt\text{ and } \int_{0}^{T-(a_1+\kappa)}\int_{\Omega}l^2(x,0,t)dxdt\leq C'_{T}\int_{\Xi'}l^{2}(x,a,t)dxdadt.\]
				Finally, adding the above inequality to the results of Proposition 2.3 and Proposition 2.2, we get:
			$$\int_{0}^{A}\int_{\Omega}n^2(x,a,0)dxda+\int_{0}^{A}\int_{\Omega}l^2(x,a,0)dxda\leq C_{T}\left(\int_{\Xi}n^2(x,a,t)dxdadt+\int_{\Xi'}l^2(x,a,t)dxdadt\right)$$	 (see Figure 2).
				\end{proof}
				\begin{figure}[H]
		\begin{overpic}[scale=0.8]{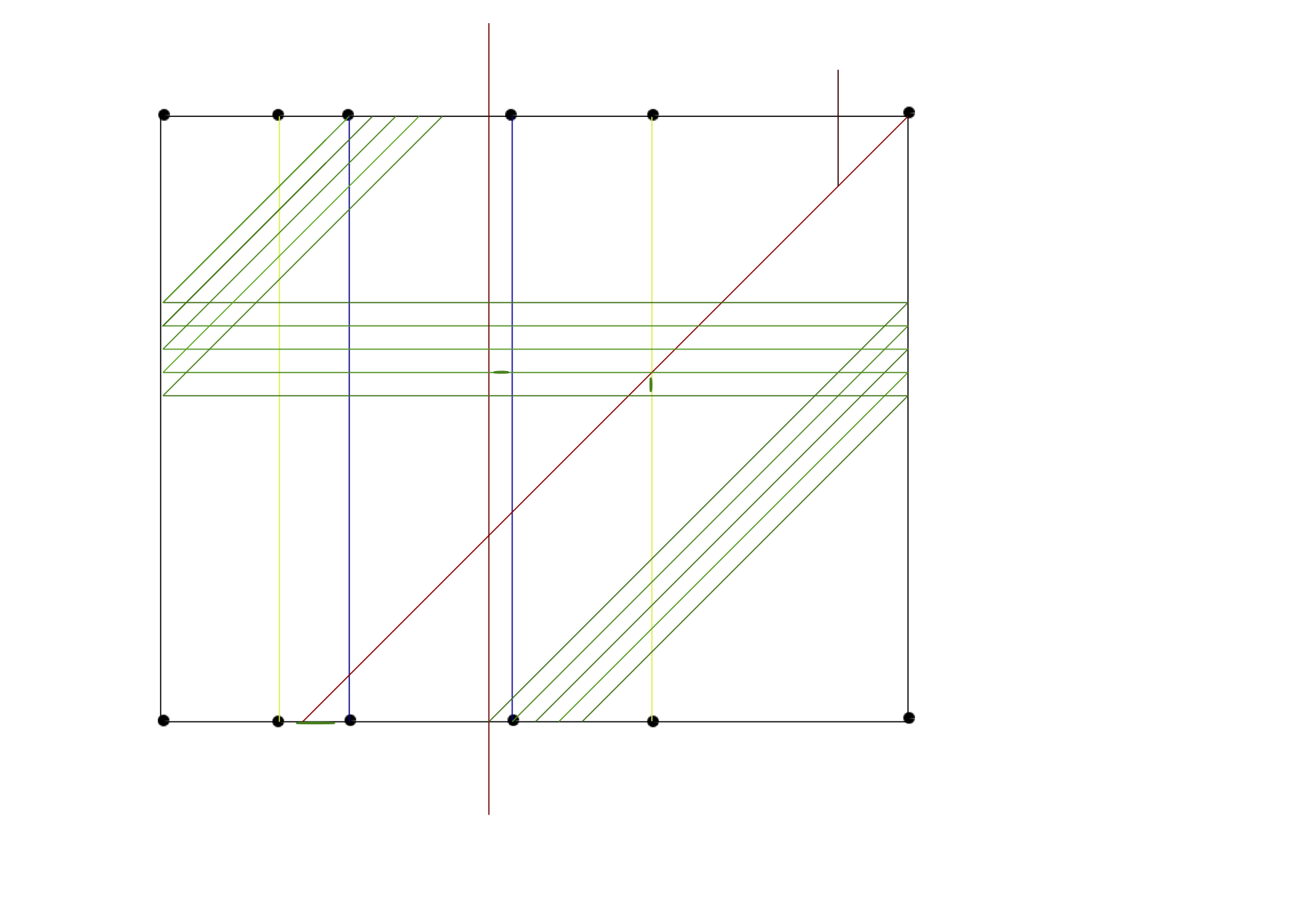}
		\put (10.5,62) {$T$}
          \put (21,13) {$b_1$}
           \put (26,13) {$a_1$}
           			\put (39,13) {$a_2=b$}	
			\put (49,13) {$b_2$}	
			\put (49,5.5) {$$}	
			\put (36.75,65) {$a_0$}	
			\put (56,66) {$\{\text{ the line }A-a=T-t\}$}		
			\put (70,13) {$A$}			
				\end{overpic}
				\caption{ Illustration of Observability inequality:\\ 
The backward characteristics starting from $(a,0)$ with
$a\in (a_2,A)$ (green lines) hits the line $(a=A)$, gets renewed
by the renewal condition $(1-\gamma)\beta(a,p)n(x,0,t)+\gamma\beta(a,p)l(x,0,t)$ and then 
enters the observation domain.\\
More precisely, the backward characteristics need at most $A-a_2$ time to hits the line $a=A$, get renewed by the renewal condition $(1-\gamma)\beta(a,p)n(x,0,t)+\gamma\beta(a,p)l(x,0,t)$ and  takes maximum  $a_1$ time to enter the observation domain. Thus, at least $T=A-a_2+a_1$ time is needed to obtain the observability inequality.\\
So with the conditions $T>A-a_2+a_1$ and $a_1<\eta<T,$ all the characteristics starting at $(a,0)$ with $a\in (a_2,A)$ get renewed by the renewal condition $(1-\gamma)\beta(a,p)n(x,0,t)+\gamma\beta(a,p)l(x,0,t)$ in $t\in (0,T-\eta)$ and enter the observation domain.}
	\end{figure}
	\begin{remark}
		The constant $C_T(T,\|\beta\|_{\infty},a_1,a_2,b_1,b_2,A,\hat{a})$ does not depend on $p$ because $\beta\in L^{\infty}((0,A)\times \mathbb{R}).$
			\end{remark}
				\subsection{Approximate null controllability result}
				We have the following result:
				\begin{theorem}
	Under assumptions $(H1)-(H2).$ For every time $T>a_1+A-a_2,$ and for every $\kappa>0\text{ and } \vartheta>0$ there exists a control $(v_{\kappa},v_{\vartheta})$ such that the solution $(m,f)$ of the system $(\ref{2})$ verifies 
	\[\|m(.,.,T)\|_{L^2(Q_A)}\leq \kappa \]
	and 
	\[\|f(.,.,T)\|_{L^2(Q_A)}\leq \vartheta. \]
	\end{theorem}			For $p\in L^2(Q_T)$, $\epsilon>0$ and $\theta>0,$ we consider the functional $J_{\epsilon,\theta}$ defined by:
	\begin{equation}
		J_{\epsilon,\theta}(v_m,v_f)=\dfrac{1}{2}\int_{\Xi}v_{m}^{2}dxdadt+\dfrac{1}{2}\int_{\Xi'}v_{f}^{2}dxdadt+\dfrac{1}{2\epsilon}\int_{0}^{A}\int_{\Omega}m^2(x,a,T)dxda+\dfrac{1}{2\theta}\int_{0}^{A}\int_{\Omega}f^2(x,a,T)dxda
	\end{equation} 
	where $(m,f)$ is the solution of the following system
	\begin{equation}
\left\{ 
\begin{array}{ccc}
m_t+m_a-K_m\Delta m+\mu_m m&=\chi_{\Xi} v_{m}&\text{ in }Q ,\\ 
f_t+f_a-K_f\Delta f+\mu_f f&=\chi_{\Xi'}v_{f}&\text{ in }Q,\\ 
m(\sigma,a,t)=f(\sigma,a,t)&=0&\text{ on }\Sigma, \\
m(x,a,0)=m_0\quad f(x,a,0)=f_0&&\text{ in }Q_A,\\
m(x,0,t)=(1-\gamma)\int_{0}^{A}\beta(a,p)fda,\quad f(x,0,t)=\gamma \int_{0}^{A}\beta(a,p)fda&& \text{ in }Q_T.\\
\end{array}
\right.
\label{21}
\end{equation}
The result of the Theorem 2.2 is obtained from the following Lemma
\begin{lemma}
The functional $J_{\epsilon,\theta}$	 are continuous, strictly convex and coercive. Consequently $J_{\epsilon,\theta}$ reaches its minimum at a point $(v_{m\epsilon},v_{f\kappa})\in L^2(\Xi)\times L^2(\Xi').$\\ 
Moreover, setting $(m_{\epsilon},f_{\theta})$ the associated solution of $(\ref{21})$ and $(n_{\epsilon},l_{\theta})$ the solution $(\ref{3})$ with $n_{\epsilon}(x,a,T)=-\dfrac{1}{\epsilon}m_{\epsilon}(x,a,T)$ and $l_{\theta}(x,a,T)=-\dfrac{1}{\theta}f_{\theta}(x,a,T)$ one has $\chi_{\Xi}v_{m,\epsilon}=\chi_{\Xi}n_{\epsilon},\quad \chi_{\Xi'}v_{f\theta}=\chi_{\Xi'}l_{\theta},$ and there exist $C_{i}>0\quad i\in\{ 1,2,3,4\}$  independent of $p$, $\epsilon$ and $\theta$, such that
\[\int_{\Xi}n^{2}_{\epsilon}dxdadt\leq C_1\left(\int_{0}^{A}\int_{\Omega}m^{2}_{0}dxda+\int_{0}^{A}\int_{\Omega}f^{2}_{0}dxda\right),\]
\[\int_{0}^{A}\int_{\Omega}m^{2}_{\epsilon}(x,a,T)dxda\leq \epsilon C_2\left(\int_{0}^{A}\int_{\Omega}m^{2}_{0}dxda+\int_{0}^{A}\int_{\Omega}f^{2}_{0}dxda\right),\]
\[\int_{\Xi'}l^{2}_{\theta}dxdadt\leq C_3\left(\int_{0}^{A}\int_{\Omega}m^{2}_{0}dxda+\int_{0}^{A}\int_{\Omega}f^{2}_{0}dxda\right)\]
and 
\[\int_{0}^{A}\int_{\Omega}f^{2}_{\theta}(x,a,T)dxda\leq \theta C_4\left(\int_{0}^{A}\int_{\Omega}m^{2}_{0}dxda+\int_{0}^{A}\int_{\Omega}f^{2}_{0}dxda\right).\]
\end{lemma}
\begin{proof}{Lemma 2.3}\\
	 It is easy to check that $J_{\epsilon,\theta}$ is coercive, continuous, and strictly convex. Then, it admits a unique minimiser $(v_{m,\epsilon},v_{f,\theta})$. The maximum principle gives that $$(\chi_{\Xi}v_{m,\epsilon},\chi_{\Xi'}v_{f,\theta})=(\chi_{\Xi}n_{\epsilon},\chi_{\Xi'} l_{\theta})$$
 where the pair $(n_{\epsilon},l_{\theta})$ is the solution of the following system:
 \begin{equation}
\left\{ 
\begin{array}{ccc}
-\dfrac{n_{\epsilon}}{\partial t}-\dfrac{n_{\epsilon}}{\partial a}-K_m\Delta n_{\epsilon}+\mu_m n_{\epsilon}&=0 &\text{ in }Q ,\\ 
-\dfrac{l_{\theta}}{\partial t}-\dfrac{l_{\theta}}{\partial a}-K_f\Delta l_{\theta}+\mu_f l_{\theta}&=(1-\gamma)\beta(a,p)n_{\epsilon}(x,0,t)+\gamma\beta(a,p)l_{\theta}(x,0,t)&\text{ in }Q,\\ 
n_{\epsilon}(\sigma,a,t)=l_{\theta}(\sigma,a,t)&=0&\text{ on }\Sigma, \\
n_{\epsilon}(x,a,T)=-\dfrac{1}{\epsilon}m_{\epsilon}(x,a,T)\quad l_{\theta}(x,a,T)=-\dfrac{1}{\theta}f_{\theta}(x,a,T)&&\text{ in }Q_A,\\
n_{\epsilon}(x,A,t)=0,\quad l_{\theta}(x,A,t)=0&& \text{ in }Q_T.\\
\end{array}
\right.
\label{3a1}
\end{equation} 
 Multiplying the first equation of $(\ref{21})$ by $n_{\epsilon}$ the second equation by $l_{\theta}$ and integrating the results over $Q,$  with $$(\chi_{\Xi}v_{m,\epsilon},\chi_{\Xi'}v_{f,\theta})=(\chi_{\Xi}n_{\epsilon},\chi_{\Xi'}l_{\Theta}),$$   we get:
 \[\int_{\Xi}n^{2}_{\epsilon}dxdadt+\dfrac{1}{\epsilon}\int_{0}^{A}\int_{\Omega}m_{\epsilon}^2(x,a,T)dxda\]\begin{equation}=\int_{0}^{A}\int_{\Omega} m_0n_{\epsilon}(x,a,0)dxda+(1-\gamma)\int_{0}^{T}\int_{0}^{A}\int_{\Omega}\beta(a,p)f n(x,0,t)dxdadt,\label{resa}\end{equation}
 and
  \[
 \dfrac{1}{\theta}\int_{0}^{A}\int_{\Omega}f_{\theta}^2(x,a,T)dxda +(1-\gamma)\int_{0}^{T}\int_{0}^{A}\int_{\Omega}\beta(a,p)f n(x,0,t)dxdadt+\int_{\Xi'}l^{2}_{\epsilon}dxdadt\]\begin{equation}
 =\int_{0}^{A}\int_{\Omega} f_0l_{\theta}(x,a,0)dxda.\label{resa1}
  \end{equation}
  Combining $(\ref{resa})$ and $(\ref{resa1}),$ we obtain:
   \[\int_{\Xi}n^{2}_{\epsilon}dxdadt+\dfrac{1}{\epsilon}\int_{0}^{A}\int_{\Omega}m_{\epsilon}^2(x,a,T)dxda+\dfrac{1}{\theta}\int_{0}^{A}\int_{\Omega}f_{\theta}^2(x,a,T)dxda+\int_{\Xi'}l^{2}_{\theta}dxdadt\]\[=\int_{0}^{A}\int_{\Omega} m_0n_{\epsilon}(x,a,0)dxd+\int_{0}^{A}\int_{\Omega} f_0l_{\epsilon}(x,a,0)dxda.\]  
     Using the inequality of Young, we obtain for any $\delta>0 $ 
 \[\int_{\Xi}n^{2}_{\epsilon}dxdadt+\dfrac{1}{\epsilon}\int_{0}^{A}\int_{\Omega}m_{\epsilon}^2(x,a,T)dxda+\dfrac{1}{\theta}\int_{0}^{A}\int_{\Omega}f_{\theta}^2(x,a,T)dxda+\int_{\Xi'}l^{2}_{\theta}dxdadt\]\[\leq \dfrac{\delta}{2}\int_{0}^{A}\int_{\Omega}m^{2}_0dxda+ \dfrac{1}{2\delta}\int_{0}^{A}\int_{\Omega}n^{2}_{\epsilon}(x,a,0)dxda+\dfrac{\delta}{2}\int_{0}^{A}\int_{\Omega}f^{2}_0dxda+ \dfrac{1}{2\delta}\int_{0}^{A}\int_{\Omega}l^{2}_{\theta}(x,a,0)dxda.\]
 Using the observability inequality $(\ref{12})$ and choosing $\delta=C_{T}$ where $C_{T}$ is given in $(\ref{12}),$ we obtain
  \[\dfrac{1}{2}\int_{\Xi}n^{2}_{\epsilon}dxdadt+\dfrac{1}{\epsilon}\int_{0}^{A}\int_{\Omega}m_{\epsilon}^2(x,a,T)dxda+\dfrac{1}{\theta}\int_{0}^{A}\int_{\Omega}f_{\theta}^2(x,a,T)dxda+\dfrac{1}{2}\int_{\Xi'}l^{2}_{\theta}dxdadt\]\[\leq\dfrac{C_{T}}{2}\left(\int_{0}^{A}\int_{\Omega}m^{2}_0dxda+\int_{0}^{A}\int_{\Omega}f^{2}_0dxda\right).\]  
  From these inequalities, we deduce the results of the Lemma 2.3 necessary to the proof of the main result. \\
  Moreover, by asking $$\kappa=\epsilon C_2\left(\int_{0}^{A}\int_{\Omega}m^{2}_{0}dxda+\int_{0}^{A}\int_{\Omega}f^{2}_{0}dxda\right),$$ $$\vartheta=\theta C_4\left(\int_{0}^{A}\int_{\Omega}m^{2}_{0}dxda+\int_{0}^{A}\int_{\Omega}f^{2}_{0}dxda\right)$$
and $$(v_{\kappa},v_{\vartheta})=(v_{m,\epsilon},v_{f,\theta}),$$ we have the desired result of the Theorem 2.2.
  \end{proof}
  Now, we consider the following system:
  	\begin{equation}
\left\{ 
\begin{array}{ccc}
\dfrac{\partial m_{\epsilon}(p)}{\partial t}+\dfrac{\partial m_{\epsilon}(p)}{\partial a}-K_m\Delta m_{\epsilon}(p)+\mu_m m_{\epsilon}(p)&=\chi_{\Xi}n_{\epsilon}(p)&\text{ in }Q ,\\ 
\dfrac{\partial f_{\theta}(p)}{\partial t}+\dfrac{\partial f_{\theta}(p)}{\partial a}-K_f\Delta f_{\theta}(p)+\mu_f f_{\theta}(p)&=\chi_{\Xi'}l_{\theta}(p)&\text{ in }Q,\\ 
m_{\theta}(p)(\sigma,a,t)=f_{\theta}(p)(\sigma,a,t)&=0&\text{ on }\Sigma, \\
m_{\epsilon}(p)(x,a,0)=m_0(x,a)\quad f_{\theta}(p)(x,a,0)=f_0&&\text{ in }Q_A,\\
m_{\epsilon}(x,0,t)=(1-\gamma)\int_{0}^{A}\beta(a,p)f_{\theta}(p)da,\quad f_{\theta}(x,0,t)=\gamma \int_{0}^{A}\beta(a,p)f_{\theta}(p)da&& \text{ in }Q_T,
\end{array}
\right.
\label{4a1}
\end{equation}
where $(n_{\epsilon}(p),l_{\theta}(p))$ is the solution of the cascade system $(\ref{3a1})$ that minimizes the functional $J_{\epsilon,\theta}.$ We have the following result:
\begin{lemma}
	Under the assumptions of the Theorem 1.1, there exists $C>0$ independent of $p$, $\epsilon$ and $\theta$ such that the solution $(m_{\epsilon}(p),f_{\theta}(p))$ verifies the following inequalities:
	\[
	\int_{0}^{A}\int_{\Omega}m_{\epsilon}^2(p)(x,a,T)dxda+\int_{0}^{T}\int_{0}^{A}\int_{\Omega}(1+\mu_m )m_{\epsilon}^2(p)dxdadt+k_m\int_{0}^{T}\int_{0}^{A}\int_{\Omega}|\nabla m_{\epsilon}(p)|^2dxdadt\]
	\begin{equation}
		\leq C\left(\int_{0}^{A}\int_{\Omega}m_{0}^{2}dxda+\int_{0}^{A}\int_{\Omega}f_{0}^{2}dxda\right)\label{xx}
		\end{equation}	
		and 
		\[
	\int_{0}^{A}\int_{\Omega}f_{\theta}^2(p)(x,a,T)dxda+\int_{0}^{T}\int_{0}^{A}\int_{\Omega}(1+\mu_f )f_{\theta}^2(p)dxdadt+K_f\int_{0}^{T}\int_{0}^{A}\int_{\Omega}|\nabla f_{\theta}(p)|^2dxdadt\]
	\begin{equation}
		\leq C\left(\int_{0}^{A}\int_{\Omega}m_{0}^{2}dxda+\int_{0}^{A}\int_{\Omega}f_{0}^{2}dxda\right).\label{ww}
		\end{equation}			
	\end{lemma} 
	\begin{proof}{of the Lemma 2.4}\\
			 Let \[(y_{\epsilon}(p),z_{\theta}(p))=(e^{-\lambda_{0}t}m_{\epsilon}(p),e^{-\gamma_0 t}f_{\theta}(p)).\]
 The function $y_{\epsilon}(p)$ and $z_{\theta}(p)$ verify 
 \begin{equation}
 	 \dfrac{\partial y_{\epsilon}(p)}{\partial t}+\dfrac{\partial y_{\epsilon}(p)}{\partial a}-K_{m}\Delta y_{\epsilon}(p)+(\lambda_0 +\mu_{m}(a))y_{\epsilon}(p)=\chi_{\Xi}e^{-\lambda_0 t}n_{\epsilon}(p)\label{3433}
 	 \end{equation}
 	 and
\begin{equation}
 	 \dfrac{\partial z_{\theta}(p)}{\partial t}+\dfrac{\partial z_{\theta}(p)}{\partial a}-k_{f}\Delta z_{\theta}(p)+(\gamma_0 +\mu_{f}(a))z_{\theta}(p)=\chi_{\Xi'}e^{-\gamma_{0}t}l_{\theta}(p).\label{344}
 	 \end{equation}

Multiplying the equality $(\ref{3433})$ and the equality $(\ref{344})$ respectively  by $y_{\epsilon}(p)$ and $z_{\theta}(p)$ and integrating on $Q$ we obtain:
\[
\dfrac{1}{2}\int_{0}^{A}\int_{\Omega}y_{\epsilon}^{2}(p)(x,a,T)dxda+\dfrac{1}{2}\int_{0}^{T}\int_{\Omega}y_{\epsilon}^{2}(p)(x,A,t)dxdt+\]\[K_m\int_{0}^{T}\int_{0}^{A}\int_{\Omega}\| \nabla y_{\epsilon}(p)\|^{2}_{L^2(\Omega)}dxdadt+\int_{0}^{T}\int_{0}^{A}\int_{\Omega}(\mu_m+\lambda_0)y_{\epsilon}^{2}(p) dxdadt\]
\begin{equation}
=\dfrac{1}{2}\|m_0\|^{2}_{L^2(\Omega\times (0,A))}+\gamma\int_{0}^{T}\int_{\Omega}\left(\int_{0}^{A}\beta(a,p)z_{\theta}(p)da\right)^2 dxdt+\int_{0}^{T}\int_{0}^{A}\int_{\Omega}\chi_{\Xi}e^{-\lambda_0 t}n_{\epsilon}(p)y_{\epsilon}(p)dxdadt\label{23c}
\end{equation}
and
\[
\dfrac{1}{2}\int_{0}^{A}\int_{\Omega}z_{\theta}^{2}(p)(x,a,T)dxda+\dfrac{1}{2}\int_{0}^{T}\int_{\Omega}z_{\theta}^{2}(p)(x,A,t)dxdt+\]\[K_f\int_{0}^{T}\int_{0}^{A}\int_{\Omega}\| \nabla z_{\theta}(p)\|^{2}_{L^2(\Omega)}dxdadt+\int_{0}^{T}\int_{0}^{A}\int_{\Omega}(\mu_f+\gamma_0)z_{\theta}^{2}(p) dxdadt
\]
\begin{equation}
=\dfrac{1}{2}\|f_0\|^{2}_{L^2(\Omega\times (0,A))}+(1-\gamma)\int_{0}^{T}\int_{\Omega}\left(\int_{0}^{A}\beta(a,p)z_{\theta}(p)da\right)^2dxdt+\int_{0}^{T}\int_{0}^{A}\int_{\Omega}\chi_{\Xi'}e^{-\gamma_0 t}l_{\theta}(p)z_{\theta}(p)dxdadt.\label{235}
\end{equation}
Using Young inequality, Cauchy Schwarz inequality and the fact that $\beta$ is $L^\infty,$ we get that:
\[\gamma\int_{0}^{T}\int_{\Omega}\left(\int_{0}^{A}\beta(a,p)z_{\theta}(p)da\right)^2dxdt+\int_{0}^{T}\int_{0}^{A}\int_{\Omega}\chi_{\Xi}n_{\epsilon}(p)y_{\epsilon}(p)dxdadt\]\[\leq A\|\beta\|^{2}_{\infty}\|z_{\theta}(p)\|^{2}_{L^2(Q)}+\dfrac{1}{2}\|y_{\epsilon}(p)\|^{2}_{L^2(Q)}+\dfrac{ 1}{2}\|n_{\epsilon}(p)\|^{2}_{L^2(\Xi)}\]
and
\[(1-\gamma)\int_{0}^{T}\int_{\Omega}\left(\int_{0}^{A}\beta(a,p)z_{\theta}(p)da\right)^2dxdt+\int_{0}^{T}\int_{0}^{A}\int_{\Omega}\chi_{\Xi}l_{\theta}(p)z_{\theta}(p)dxdadt\]\[\leq A\|\beta\|^{2}_{\infty}\|z_{\theta}(p)\|^{2}_{L^2(Q)}+\dfrac{1}{2}\|z_{\theta}(p)\|^{2}_{L^2(Q)}+\dfrac{ 1}{2}\|l_{\theta}(p)\|^{2}_{L^2(\Xi')}.\]
Therefore, choosing $\gamma_0=(A\|\beta\|^{2}_{\infty}+3/2),$ we get:
\[
	\int_{0}^{A}\int_{\Omega}z_{\theta}^2(p)(x,a,T)dxda+\int_{0}^{T}\int_{0}^{A}\int_{\Omega}(1+\mu_f )z_{\theta}^2(p)dxdadt+K_f\int_{0}^{T}\int_{0}^{A}\int_{\Omega}|\nabla z_{\theta}(p)|^2dxdadt\]
	\begin{equation}
		\leq 	\dfrac{1}{2}\|f_0\|^{2}_{L^2(\Omega\times (0,A))}+\dfrac{ 1}{2}\|l_{\theta}(p)\|^{2}_{L^2(\Xi')}	.
		\end{equation}	
		Finally, applying the result of the Lemma 2.3 in the above inequality, we get:
		\[
	\int_{0}^{A}\int_{\Omega}z_{\theta}^2(p)(x,a,T)dxda+\int_{0}^{T}\int_{0}^{A}\int_{\Omega}(1+\mu _f)z_{\theta}^2(p)dxdadt+K_f\int_{0}^{T}\int_{0}^{A}\int_{\Omega}|\nabla z_{\theta}|^2dxdadt\]
	\begin{equation}
		\leq C\left(\int_{0}^{A}\int_{\Omega}m_{0}^{2}dxda+\int_{0}^{A}\int_{\Omega}f_{0}^{2}dxda\right)	\label{qqq}	
		\end{equation}			and then the inequality $(\ref{ww})$ holds.\\
		Likewise, choosing $\lambda_0=3/2$, we get 
			\[
	\int_{0}^{A}\int_{\Omega}y_{\epsilon}^2(p)(x,a,T)dxda+\int_{0}^{T}\int_{0}^{A}\int_{\Omega}(1+\mu_m )y_{\epsilon}^2(p)dxdadt+K_m\int_{0}^{T}\int_{0}^{A}\int_{\Omega}|\nabla y_{\epsilon}(p)|^2dxdadt\]
	\begin{equation}
		\leq\dfrac{1}{2}\|m_0\|^{2}_{L^2(\Omega\times (0,A))}+\|\beta\|^{2}_{\infty}\|z\|^{2}_{L^2(Q)}+\dfrac{ 1}{2}\|n_{\epsilon}(p)\|^{2}_{L^2(\Xi)}.\label{xxx}
		\end{equation}		
			Using the Lemma 2.3 and the inequality $(\ref{qqq})$ we obtain
				\[
	\int_{0}^{A}\int_{\Omega}y_{\epsilon}^2(p)(x,a,T)dxda+\int_{0}^{T}\int_{0}^{A}\int_{\Omega}(1+\mu_m )y_{\epsilon}^2(p)dxdadt+K_m\int_{0}^{T}\int_{0}^{A}\int_{\Omega}|\nabla y_{\epsilon}(p)|^2dxdadt\]
	\begin{equation}
		\leq C\left(\int_{0}^{A}\int_{\Omega}m_{0}^{2}dxda+\int_{0}^{A}\int_{\Omega}f_{0}^{2}dxda\right).
		\label{xxxx}
		\end{equation} 
	As $(y_{\epsilon}(p),z_{\theta}(p))=(e^{-\lambda_{0}t}m_{\epsilon}(p),e^{-\gamma_0 t}f_{\theta}(p)),$ we deduce from $(\ref{qqq})$ and $(\ref{xxxx})$ the result $(\ref{xx})$ and $(\ref{ww}).$
			\end{proof}
				We have now the necessary ingredients for the proof of Theorem 1.1.
	\section{Proof of Theorem 1.1}
	In the following lines, we establish the existence of a fixed point for the previous auxiliary problem and we consider the limit when $(\epsilon,\theta)$ goes to $(0,0)$. Indeed, we consider that $(H_3)$ holds and we suppose for a sake of simplicity that $\lambda(0)=\lambda(A)=0.$
		Let's define now the operator
		\[\Lambda: L^{2}(Q_T)\longrightarrow L^{2}(Q_T)\quad p\longmapsto \int_{0}^{A}\lambda (a)m_{\epsilon}(p)da\]
		where the pair $(m_{\epsilon}(p),f_{\epsilon}(p))$ is the solution of the following cascade system:
		\begin{equation}
\left\{ 
\begin{array}{ccc}
\dfrac{\partial m_{\epsilon}(p)}{\partial t}+\dfrac{\partial m_{\epsilon}(p)}{\partial a}-K_m\Delta m_{\epsilon}(p)+\mu_m m_{\epsilon}(p)&=\chi_{\Xi}n_{\epsilon}(p)&\text{ in }Q ,\\ 
\dfrac{\partial f_{\theta}(p)}{\partial t}+\dfrac{\partial f_{\theta}(p)}{\partial a}-K_f\Delta f_{\theta}(p)+\mu_f f_{\theta}(p)&=\chi_{\Xi'}l_{\theta}(p)&\text{ in }Q,\\ 
m_{\theta}(p)(\sigma,a,t)=f_{\theta}(p)(\sigma,a,t)&=0&\text{ on }\Sigma, \\
m_{\epsilon}(p)(x,a,0)=m_0\quad f_{\theta}(p)(x,a,0)=f_0&&\text{ in }Q_A,\\
m_{\epsilon}(p)(x,0,t)=(1-\gamma)\int_{0}^{A}\beta(a,p)f_{\theta}(p)da,\quad f_{\theta}(p)(x,0,t)=\gamma \int_{0}^{A}\beta(a,p)f_{\theta}(p)da&& \text{ in }Q_T,
\end{array}
\right.
\label{8}
\end{equation}
and $(n_{\epsilon}(p),l_{\theta}(p))$ the corresponding  minimizer of $J_{\epsilon,\theta}.$\\
We have the following result:
\begin{proposition}
	Under the assumption of the Theorem 1.1, the operator $\Lambda$ is continuous, bounded, and compact on $L^{2}(Q_T)$. Then $\Lambda$ admits a fixed point.
\end{proposition}
\begin{proof}{of the Proposition 3.1}\\
The proof will be done in two steps.\\
$$\textbf{Step1: Boundedness and compactness of $\Lambda$}.$$
Let $Y(x,t)=\int_{0}^{A}\lambda(a)m_{\epsilon}(p) da.$ It easy to prove that $Y$ is solution of the following system:
\begin{equation}
\left\{ 
\begin{array}{ccc}
Y_t-K_f\Delta Y+\int_{0}^{A}\mu_m\lambda m_{\epsilon}(p)da&=R(x,t)&\text{ in }Q_T,\\ 
Y(\sigma,t)&=0&\text{ on }\Sigma_T, \\
Y(x,0)=\int_{0}^{A}\lambda(a)m_0da&&\text{ in }\Omega,\\
\end{array}
\right.
\label{2222}
\end{equation}
where \[R(x,t)=\int_{0}^{A}\lambda'(a)m_{\epsilon}(p)da+\int_{0}^{A}\chi_{\Xi}\lambda(a)n_{\epsilon}(p)da.\]
Using the results of the Lemma 2.4, the results of the Theorem 2.1 and the assumption on $\lambda$, there exists, $K>0$ such that 
\[\|R\|_{L^2(Q_T)}\leq K\left(\|m_0\|_{L^2(Q_A)}+\|f_0\|_{L^2(Q_A)}\right)\]
and $Y(.,0)\in L^2(\Omega).$\\
$\textbf{Boundedness of } S(x,t)=\int_{0}^{A}\mu_m\lambda m_{\epsilon}(p)da \text{ in } L^2(Q_T).$\\ 

$$\int_{0}^{T}\int_{\Omega}\left(\int_{0}^{A}\mu_m\lambda m_{\epsilon}(p) da\right)^2dxdt=\int_{0}^{T}\int_{\Omega}\left(\int_{0}^{A}\left(\mu_m\lambda\right)^{\frac{1}{2}}\left((\mu_m\lambda)^{\frac{1}{2}}m_{\epsilon}(p)\right) da\right)^2dxdt.$$ Using the Cauchy Swartz inequality, we obtain 
$$\int_{0}^{T}\int_{\Omega}\left(\int_{0}^{A}\mu_m\lambda m_{\epsilon}(p) da\right)^2dxdt\leq \int_{0}^{A}\mu_m\lambda da \int_{0}^{T}\int_{\Omega}\int_{0}^{A}\mu_m\lambda m^{2}_{\epsilon}(p)dadxdt .$$ 
The Lemma 2.4 and the fact that $\lambda\in C([0,A])$ give that  $$\int_{0}^{T}\int_{\Omega}\int_{0}^{A}\mu_m\lambda m^{2}_{\epsilon}(p)dadxdt<R_1\left(\|m_0\|^{2}_{L^2(Q_A)}+\|f_0\|^{2}_{L^2(Q_A)}\right),$$ where $R_1>0$ is independent of $p,$ $\epsilon$ and $\theta.$
Moreover, as $\lambda\mu_m\in L^{1}(0,A),$ then
$\int_{0}^{A}\lambda\mu_m m_{\epsilon}(p)da$ is bounded in $L^2(Q_T)$ independently of $p,$ $\epsilon$ and $\theta.$\\
Therefore $Y$ is bounded in $L^2(0,T;H^{1}_{0}(\Omega))$ and $\dfrac{\partial Y}{\partial t}$ is bounded in $L^2(0,T;H^{-1}(\Omega))$ independently of $p,$ $\epsilon$ and $\theta.$\\
 	 Hence, using the Lions-Aubin Lemma we conclude that $\Lambda$ is bounded and compact in $L^2(Q_T).$\\
 	 $$\textbf{Step2: Continuity of the operator $\Lambda$}.$$
 	 Let $(m_{1,\epsilon}(p),f_{2,\theta})(p)$ and $(m_{2,\epsilon}(p),f_{1,\theta})(p)$ the solution of the following cascade system
 	  	  	\begin{equation}
\left\{ 
\begin{array}{ccc}
\dfrac{\partial m_{1,\epsilon}(p)}{\partial t}+\dfrac{\partial m_{1,\epsilon}(p)}{\partial a}-K_m\Delta m_{1,\epsilon}(p)+\mu_m m_{1,\epsilon}(p)&=\chi_{\Xi}v_{m,\epsilon}&\text{ in }Q ,\\ 
\dfrac{\partial f_{1,\theta}(p)}{\partial t}+\dfrac{\partial f_{1,\theta}(p)}{\partial a}-K_f\Delta f_{1,\theta}(p)+\mu_f f_{1,\theta}(p)&=\chi_{\Xi'}v_{f,\theta}&\text{ in }Q,\\ 
m_{1,\theta}(p)(\sigma,a,t)=f_{1,\theta}(p)(\sigma,a,t)&=0&\text{ on }\Sigma, \\
m_{1,\epsilon}(p)(x,a,0)=0\quad f_{1,\theta}(p)(x,a,0)=0&&\text{ in }Q_A,\\
m_{1,\epsilon}(p)(x,0,t)=(1-\gamma)\int_{0}^{A}\beta(a,p)f_{1,\theta}(p)da,\quad f_{1,\theta}(p)(x,0,t)=\gamma \int_{0}^{A}\beta(a,p)f_{1,\theta}(p)da&& \text{ in }Q_T,
\end{array}
\right.
\label{88}
\end{equation}
and
	\begin{equation}
\left\{ 
\begin{array}{ccc}
\dfrac{\partial m_{2,\epsilon}(p)}{\partial t}+\dfrac{\partial m_{2,\epsilon}(p)}{\partial a}-K_m\Delta m_{2,\epsilon}(p)+\mu_m m_{2,\epsilon}(p)&=0&\text{ in }Q ,\\ 
\dfrac{\partial f_{2,\theta}(p)}{\partial t}+\dfrac{\partial f_{2,\theta}(p)}{\partial a}-K_f\Delta f_{2,\theta}(p)+\mu_f f_{2,\theta}(p)&=0&\text{ in }Q,\\ 
m_{2,\theta}(p)(\sigma,a,t)=f_{2,\theta}(p)(\sigma,a,t)&=0&\text{ on }\Sigma, \\
m_{2,\epsilon}(p)(x,a,0)=m_0\quad f_{2,\theta}(p)(x,a,0)=f_0&&\text{ in }Q_A,\\
m_{2,\epsilon}(p)(x,0,t)=(1-\gamma)\int_{0}^{A}\beta(a,p)f_{2,\theta}(p)da,\quad f_{2,\theta}(p)(x,0,t)=\gamma \int_{0}^{A}\beta(a,p)f_{2,\theta}(p)da&& \text{ in }Q_T,
\end{array}
\right..
\label{89}
\end{equation}
Due of the Lemma 2.3, there exists $(v_{1,\epsilon},v_{2,\theta})\in L^{2}(\Xi)\times L^{2}(\Xi')$ such that
\begin{equation}\|m_{1,\epsilon}(p)(.,.,T)+m_{2,\epsilon}(p)(.,.,T)\|\leq \kappa\text{ and }\|f_{1,\theta}(p)(.,.,T)+f_{2,\theta}(p)(.,.,T)\|\leq \vartheta \label{eere}\end{equation}
 for any $\epsilon>0$ and $\theta>0.$\\
Any suitable control can be chosen, in particular the control of minimum norm in $L^{2}(\Xi)\times L^{2}(\Xi').$  
Let us then consider 
$$m_{\epsilon}(p)=m_{1,\epsilon}(p)+m_{2,\epsilon}(p)\text{ and }f_{\theta}(p)=f_{1,\theta}(p)+f_{2,\theta}(p)$$
where $(v_{1,\epsilon},v_{2,\theta})$ is the optimal control characterised by the Lemma 2.3 such that $(\ref{eere})$ hold and let us consider the operator \begin{equation}
	\Lambda':\left\{
	\begin{array}{ccc}
		L^{2}(Q_T)\rightarrow L^{2}(Q)\\
		p\longmapsto m_{\epsilon}(p)
	\end{array}
	\right.
\end{equation}
The function $\beta(a.p)$ is bounded in $L^{\infty}(Q)$ when $a\in [0,A]$ and $p$ spans $L^2(Q_T).$ Since the solution $(m_{1,\epsilon},f_{1,\theta})$ on the system $(\ref{88})$
 	 depends continuously on the data $(v_{1,\epsilon},v_{2,\theta})$ it follows that the operator $\Lambda'$ is continuous. Moreover, the operator $\Lambda''\left(l\right)=\int_{0}^{A}\lambda(a)l da$ is continuous on $L^2(Q)$ and therefore we conclude that the operator $\Lambda$ is continuous.\\
 	 Since the operator $\Lambda$ is continuous, bounded, and compact on $L^2(Q_T)$ onto $L^2(Q_T)$, Schauder's fixed-point theorem
implies that $\Lambda$ admits a fixed point.\\
\end{proof}
So we get:
\begin{equation}
\left\{ 
\begin{array}{ccc}
\dfrac{\partial m_{\epsilon}}{\partial t}+\dfrac{\partial m_{\epsilon}}{\partial a}-K_m\Delta m_{\epsilon}+\mu_m m_{\epsilon}&=\chi_{\Xi}n_{\epsilon}&\text{ in }Q ,\\ 
\dfrac{\partial f_{\theta}}{\partial t}+\dfrac{\partial f_{\theta}}{\partial a}-K_f\Delta f_{\theta}+\mu_f f_{\theta}&=\chi_{\Xi'}l_{\theta}&\text{ in }Q,\\ 
m_{\theta}((\sigma,a,t)=f_{\theta}(\sigma,a,t)&=0&\text{ on }\Sigma, \\
m_{\epsilon}(x,a,0)=m_0\quad f_{\theta}(x,a,0)=f_0&&\text{ in }Q_A,\\
m(x,0,t)=(1-\gamma)\int_{0}^{A}\beta\left(a,\int_{0}^{A}\lambda(a)m_{\epsilon}da\right)f_{\theta}da,\quad f_{\theta}(x,0,t)=\gamma \int_{0}^{A}\beta\left(a,\int_{0}^{A}\lambda(a)m_{\epsilon}da\right)f_{\theta}da&& \text{ in }Q_T
\end{array}
\right.
\label{8}
\end{equation}
with $$\|m(.,.T)\|_{L^2(Q_A)}<\kappa\text{ and }\|f(.,.T)\|_{L^2(Q_A)}<\vartheta.$$
Finally, if $(\kappa,\vartheta)\longrightarrow (0, 0)$ we get:
 $$(\chi_{\Xi}n_{\epsilon},\chi_{\Xi'}l_{\theta})\rightharpoonup (\chi_{\Xi}v_m,\chi_{\Xi'}v_f)\text{ and }(m_{\epsilon},f_{\theta})\rightharpoonup (m,f)$$ with $(m,f)$ solution of the system $(\ref{1})$
and then we obtain the null controllability of the system $(\ref{1}).$
This achieves the proof of the Theorem 1.1.
\section{Proof of the Theorem 1.2}
\subsection{Proof of the Theorem 1.2-(1)}
\subsubsection{Observability inequality}
In this section, we always consider the following system:
\begin{equation}
\left\{ 
\begin{array}{ccc}
m_t+m_a-K_m\Delta m+\mu_m m&=\chi_{\Theta}v_{m}&\text{ in }Q ,\\ 
f_t+f_a-K_f\Delta f+\mu_f f&=0&\text{ in }Q,\\ 
m(\sigma,a,t)=f(\sigma,a,t)&=0&\text{ on }\Sigma, \\
m(x,a,0)=m_0\quad f(x,a,0)=f_0&&\text{ in }Q_A,\\
m(x,0,t)=(1-\gamma)\int_{0}^{A}\beta(a,p)fda,\quad f(x,0,t)=\gamma \int_{0}^{A}\beta(a,p)fda&& \text{ in }Q_T.\\
\end{array}
\right.
\label{bbb}
\end{equation}	
For every $p$ in $L^2(Q_T),$ under the assumptions of the Theorem 1.2; the controllability problem that is to find $v_m\in L^2(\Theta)$ such that $(m,f)$ solution of the system $(\ref{bbb})$
verifies
$$m(.,.,T)=0\quad x\in \Omega \text{ and } a\in (\varrho,A)$$
is equivalent to the observability inequality stated below.
\begin{proposition}
	Let us assume true the assumption $(H_1)-(H_2)-(H_3),$ for every $T>A-a_2$ and for any $\varrho>0,$ if $h_T(x,a)=0\text{ a.e in } \Omega\times (0,\varrho),$ there exists $C_{\varrho,T}>0$ such that the following inequality 
\begin{equation}
	\int_{0}^{A}\int_{\Omega}h^{2}(x,a,0)dxda+\int_{0}^{A}\int_{\Omega}g^{2}(x,a,0)dxda\leq C_{\varrho,T}\int_{\Theta}h^{2}(x,a,t)dxdadt \label{1000}
\end{equation}
holds, where $(h,g)$ is the solution of the following system
\begin{equation}
\left\{ 
\begin{array}{ccc}
-h_t-h_a-K_m\Delta h+\mu_m h&=0&\text{ in }Q ,\\ 
-g_t-g_a-K_f\Delta g+\mu_f g&=(1-\gamma)\beta(a,p)g(x,0,t)+\gamma\beta(a,p)h(x,0,t)&\text{ in }Q ,\\
h(\sigma,a,t)=g(\sigma,a,t)&=0&\text{ on }\Sigma, \\
h(x,a,T)=h_T\text{, } g(x,a,T)=0&&\text{ in }Q_A,\\
h(x,A,t)=0\text{, }g(x,A,t)=0&& \text{ in }Q_T.\\
\end{array}
\right.
\label{ccc1}
\end{equation}	
\end{proposition}
For the Proof of the Proposition 4.1, we also need the following estimate:
		\begin{proposition}
Under 	the assumption $(H_1)$ and $(H_2),$ there exists a constant $C$ such that the pair $(h,g)$ solution of the system $(\ref{ccc1})$ verifies
\begin{equation}
	\int_{0}^{A}\int_{\Omega}g^2(x,a,0)dxda+\int_{0}^{T}\int_{0}^{A}\int_{\Omega}(1+\mu )g^2dxdadt+\int_{0}^{T}\int_{0}^{A}\int_{\Omega}|\nabla g|^2dxdadt\leq C\int_{0}^{T}\int_{\Omega}h^2(x,0,t)dxdt.
		\end{equation}	
		Moreover, we deduce that for $l_T=0\text{ a.e in }\Omega\times(0,\varrho),$ there exists a constant $C_{\varrho,T}$ such that
		$$\int_{0}^{A}\int_{\Omega}g^2(x,a,0)dxda\leq C_{\varrho,T}\int_{0}^{T}\int_{\Omega}h^2(x,0,t)dxdt\leq C_{\varrho,T}\int_{0}^{T}\int_{0}^{a_2}\int_{\omega}h^2(x,a,t)dxdadt$$		
		\end{proposition}
		\begin{proof}{of Proposition 4.2}\\
					 We put by \[y=e^{\lambda_{0}t}g.\]
 The function $y$ verifies
 \begin{equation}
 	 -\dfrac{\partial y}{\partial t}-\dfrac{\partial y}{\partial a}-\Delta y+(\lambda_0 +\mu(a))y=\gamma\beta(a,p)e^{\lambda_0 t}h(x,0,t)+(1-\gamma)\beta(a,p)y(x,0,t).\label{34}
 	 \end{equation}
Multiplying the equality $(\ref{34})$ by $y$ and integrating on $Q$, we obtain:
\[
\dfrac{1}{2}\int_{0}^{A}\int_{\Omega}y^{2}(x,a,0)dxda+\dfrac{1}{2}\int_{0}^{T}\int_{\Omega}y^{2}(x,0,t)dxdt+\int_{0}^{T}\int_{0}^{A}\int_{\Omega}\| \nabla y\|^{2}_{L^2(\Omega)}dadt+\int_{0}^{T}\int_{0}^{A}\int_{\Omega}(\mu(a)+\lambda_0)y^{2} dxdadt
\]
\begin{equation}
=\int_{0}^{T}\int_{0}^{A}\int_{\Omega} \left(\gamma\beta(a,p)e^{\lambda_0 t}h(x,0,t)+(1-\gamma)\beta(a,p)y(x,0,t)\right)y dxdadt.\label{23}
\end{equation}
Using the Young inequality and the condition on $\beta,$ we get that:
\[\int_{0}^{T}\int_{0}^{A}\int_{\Omega} \left(\gamma\beta(a,p)e^{\lambda_0 t}h(x,0,t)+(1-\gamma)\beta(a,p)y(x,0,t)\right)y dxdadt\]\[\leq\dfrac{ e^{2\lambda_0 T}\|\beta\|^{2}_{\infty}}{2\delta}\|h(.,0,.)\|^{2}_{L^2(Q_T)}+\dfrac{\delta}{2}\|y\|^{2}_{L^2(Q)}+\dfrac{ e^{2\lambda_0 T}\|\beta\|^{2}_{\infty}}{2\delta}\|g(.,0,.)\|^{2}_{L^2(Q_T)}+\dfrac{\delta}{2}\|y\|^{2}_{L^2(Q)}.\]
Therefore, choosing $\delta=e^{2\lambda_0 T}\|\beta\|^{2}_{\infty},$ we obtain:
\[
\dfrac{1}{2}\int_{0}^{A}\int_{\Omega}y^{2}(x,a,0)dxda+\int_{0}^{T}\int_{0}^{A}\int_{\Omega}\| \nabla y\|^{2}_{L^2(\Omega)}dadt+\int_{0}^{T}\int_{0}^{A}\int_{\Omega}(\mu(a)+\lambda_0)y^{2} dxdadt\]\[\leq \dfrac{1}{2}\|h(.,0,.)\|^{2}_{L^2(Q_T)}+e^{2\lambda_0 T}\|\beta\|^{2}_{\infty}\|y\|^{2}_{L^2(Q)}.\]
Choosing $\lambda_0>e^{2\lambda_0 T}\|\beta\|^{2}_{\infty}+1$, we get
\[
\dfrac{1}{2}\int_{0}^{A}\int_{\Omega}y^{2}(x,a,0)dxda+\int_{0}^{T}\int_{0}^{A}\int_{\Omega}\| \nabla y\|^{2}_{L^2(\Omega)}dadt+\int_{0}^{T}\int_{0}^{A}\int_{\Omega}(\mu(a)+\lambda_0)y^{2} dxdadt\]
\[\leq \dfrac{1}{2}\int_{0}^{T}\int_{\Omega}h^2(x,0,t)dxdt.\]
Finally, we have \[
\dfrac{1}{2}\int_{0}^{A}\int_{\Omega}g^{2}(x,a,0)dxda+\int_{0}^{T}\int_{0}^{A}\int_{\Omega}\| \nabla g\|^{2}_{L^2(\Omega)}dadt+\int_{0}^{T}\int_{0}^{A}\int_{\Omega}(\mu(a)+\lambda_0)g^{2} dxdadt\]
\[\leq C_T\int_{0}^{T}\int_{\Omega}h^2(x,0,t)dxdt.\]
From the Proposition 4.2, we have:
	\begin{equation}
	\int_{0}^{A}\int_{\Omega}g^2(x,a,0)dxda\leq C\int_{0}^{T}\int_{\Omega}h^2(x,0,t)dxdt.\label{13}
	\end{equation}
		Combining the inequality $(\ref{13})$ and the inequality $(\ref{14})$ of the Proposition 2.1, for $h_T(x,a)=0\text{ a.e in } \Omega\times (0,\varrho),$ we obtain,
	\begin{equation}
	\int_{0}^{A}\int_{\Omega}g^2(x,a,0)dxda\leq C_{\varrho,T}\int_{0}^{T}\int_{0}^{a_2}\int_{\omega}h^2(x,a,t)dxdadt.\label{11}
	\end{equation}
	\end{proof}
	We are ready now to prove Proposition 4.1 
		\begin{proof}{of the Proposition 4.1}\\
	We use now the results of Proposition 2.2 and Proposition 4.2.\\
			Indeed, combining $(\ref{10})$ and $(\ref{11}),$ we get 
		\begin{equation}
	\int_{0}^{A}\int_{\Omega}g^2(x,a,0)dxda+\int_{0}^{A}\int_{\Omega}h^2(x,a,0)dxda\leq C_{\varrho,T}\int_{0}^{T}\int_{0}^{a_2}\int_{\omega}h^2(x,a,t)dxdadt.\label{120}
	\end{equation}
		\end{proof}
		 \subsubsection{Null controllability of the auxiliary system and proof of Theorem 1 2-(1)}
Now, let $\epsilon>0$ and $\varrho>0.$\\ 
We consider the functional $J_{\epsilon}$ defined by:
	\begin{equation}
		J_{\epsilon}(v_m)=\dfrac{1}{2}\int_{0}^{T}\int_{a_1}^{a_2}\int_{\omega}v_{m}^{2}dxdadt+\dfrac{1}{2\epsilon}\int_{\varrho}^{A}\int_{\Omega}m^2(x,a,T)dxda
	\end{equation} 
	where $(m,f)$ is the solution of the following system
	\begin{equation}
\left\{ 
\begin{array}{ccc}
m_t+m_a-K_m\Delta m+\mu_m m&=\chi_{\Theta} v_m&\text{ in }Q ,\\ 
f_t+f_a-K_f\Delta f+\mu_f f&=0&\text{ in }Q,\\ 
m(\sigma,a,t)=f(\sigma,a,t)&=0&\text{ on }\Sigma, \\
m(x,a,0)=m_0\quad f(x,a,0)=f_0&&\text{ in }Q_A,\\
m(x,0,t)=(1-\gamma)\int_{0}^{A}\beta(a,p)fda,\quad f(x,0,t)=\gamma \int_{0}^{A}\beta(a,p)fda&& \text{ in }Q_T.\\
\end{array}
\right.
\label{210}
\end{equation}
\begin{lemma}
The functional $J_{\epsilon}$ is continuous, strictly convex and coercive. Consequently $J_{\epsilon}$ reaches its minimum at a point $v_{m,\epsilon}\in L^2(\Theta).$\\ 
Moreover, setting $m_{\epsilon}$ the associated solution of $(\ref{210})$ and $n_{\epsilon}$ the solution $(\ref{ccc1})$ with $h_{\epsilon}(x,a,T)=-\dfrac{1}{\epsilon}\chi_{\{(\varrho,A)\}}m_{\epsilon}(x,a,T)$ one has $v_{m,\epsilon}=\chi_{\Theta}h_{\epsilon}$ and there exist $C_1>0$  and $C_2>0$ independent of $\epsilon$, such that
\[\int_{0}^{T}\int_{0}^{a_2}\int_{\omega}h^{2}_{\epsilon}dxdadt\leq C_1\left(\int_{0}^{A}\int_{\Omega}m^{2}_{0}(x,a)dxda+\int_{0}^{A}\int_{\Omega}f^{2}_{0}(x,a)dxda\right)\]
and 
\[\int_{\varrho}^{A}\int_{\Omega}m^{2}_{\epsilon}(x,a,T)dxda\leq \epsilon C_2\left(\int_{0}^{A}\int_{\Omega}m^{2}_{0}dxda+\int_{0}^{A}\int_{\Omega}f^{2}_{0}dxda\right).\]
\end{lemma}
\begin{proof}{of Lemma 4.1}\\
	The proof of the Lemma 4.1 is similar to that of Lemma 2.3.
\end{proof} 
By making $\epsilon$ tending towards zero, we thus obtain that.
\\
$$\chi_{\Theta} h_{\epsilon}\rightharpoonup \chi_{\Theta} v_{m}\text{ and }(m_{\epsilon},f_{\epsilon})\rightharpoonup (m,f)$$ where $(m,f)$ is the solution of the system $(\ref{210})$ that verifies $$m(x,a,T)=0\text{ a.e in }(x,a)\in \Omega\times (\varrho, A).$$
 Let us define now the operator
		\[\Lambda: L^{2}(\Omega\times (0,T))\rightarrow L^{2}(\Omega\times (0,T))\quad p\longmapsto \int_{0}^{A}\lambda (a)m(p)da\]
		where the pair $(m(p),f(p))$ is the solution of the following cascade system:
		\begin{equation}
\left\{ 
\begin{array}{ccc}
\dfrac{\partial m(p)}{\partial t}+\dfrac{\partial m(p)}{\partial a}-K_m\Delta m(p)+\mu_m m(p)&=\chi_{\Theta}v_{m}(p)&\text{ in }Q ,\\ 
\dfrac{\partial f(p)}{\partial t}+\dfrac{\partial f(p)}{\partial a}-K_f\Delta f(p)+\mu_f f(p)&=0&\text{ in }Q,\\ 
m(p)(\sigma,a,t)=f_{\epsilon}(p)(\sigma,a,t)&=0&\text{ on }\Sigma, \\
m(p)(x,a,0)=m_0(x,a)\quad f(p)(x,a,0)=f_0&&\text{ in }Q_A,\\
m(x,0,t)=(1-\gamma)\int_{0}^{A}\beta(a,p)f(p)da,\quad f(x,0,t)=\gamma \int_{0}^{A}\beta(a,p)f(p)da&& \text{ in }Q_T,\\
\end{array}
\right.
\label{8777}
\end{equation}
with $v_m(p)$ such that the solution $(m(p),f(p))$ verifies $$m(p)(x,a,T)=0\text{ a.e in }(x,a)\in \Omega\times (\varrho, A).$$
Let $Y(x,t)=\int_{0}^{A}\lambda(a)m(p) da.$ It easy to prove that $Y$ is solution of the following system:
\begin{equation}
\left\{ 
\begin{array}{ccc}
Y_t-K_f\Delta Y+\int_{0}^{A}\mu_m(a)\lambda(a) m(p)da&=R(x,t)&\text{ in }Q_T,\\ 
Y(\sigma,t)&=0&\text{ on }\Sigma_T, \\
Y(x,0)=\int_{0}^{A}\lambda(a)m_0da&&\text{ in }\Omega,\\
\end{array}
\right.
\label{222}
\end{equation}
where \[R(x,t)=\int_{0}^{A}\lambda'(a)m(p)da+\int_{0}^{a_2}\lambda(a)v_{m}(p)da.\]
 By the Schauder's fixed point theorem it easy to prove that $\Lambda$ admits a fixed point which gives the result of the Theorem 1.2-1.\\

\subsection{Proof of the Theorem 1.2-(2)}
Let $p\in L^2(Q_T),$ under the assumptions of the Theorem 1.2, the following controllability problem find $v_f\in L^2(\Theta)$ such that the solution of the system:
\begin{equation}
\left\{ 
\begin{array}{ccc} 
f_t+f_a-K_f\Delta f+\mu_f f&=\chi_{\Theta'}v_f&\text{ in }Q,\\ 
f(\sigma,a,t)&=0&\text{ on }\Sigma, \\
f(x,a,0)=f_0&&\text{ in }Q_A,\\
 f(x,0,t)=\gamma \int_{0}^{A}\beta(a,p)fda&& \text{ in }Q_T.\\
\end{array}
\right.
\label{bba}
\end{equation}
verifies $$f(x,a,T)=0\text{ a.e } x\in \Omega\quad a\in (0,A)$$
is equivalent to the following observability inequality:
\begin{proposition}
	Let us assume true the assumption $(H_1)-(H_2)-(H_3).$ For any $T>a_1+A-a_2$ there exists $C_T>0$ such that
\begin{equation}
	\int_{0}^{A}\int_{\Omega}g^{2}(x,a,0)dxda\leq C_T\int_{\Xi}g^{2}dxdadt ,\label{100}
\end{equation}
where $g$ is the solution of the following system
\begin{equation}
\left\{ 
\begin{array}{ccc}
-g_t-g_a-K_m\Delta g+\mu_m g&=\beta(a,p)g(x,0,t)&\text{ in }Q ,\\ 
g(\sigma,a,t)&=0&\text{ on }\Sigma, \\
g(x,a,T)=g_T&&\text{ in }Q_A,\\
g(x,A,t)=0,&& \text{ in }Q_T.\\
\end{array}
\right.
\label{ccca}
\end{equation}	
\end{proposition}
\begin{proof}{of the Proposition 4.3}\\
	Using the inequality $(14)$ of the Proposition 2.2, the result of the Proposition 2.5 and the representation of the solution of the system $(\ref{ccca}),$ we obtain the desired result.
\end{proof}
Now we consider the following operator
$$N:L^2(Q_T)\rightarrow L^2(Q_T) \quad p\longmapsto v_f(p)\longmapsto (p,f(v_f(p))\longmapsto m(p,f(v_f(p),p))\longmapsto \int_{0}^{A}\lambda(a)m(p,f(v_f(p),p))da,$$ where $(m(p,f(v_f(p))),f(p,v_f(p)))$ is the solution of the following system
\begin{equation}
\left\{ 
\begin{array}{ccc}
m_t+m_a-K_m\Delta m+\mu_m m&=0&\text{ in }Q ,\\ 
f_t+f_a-K_f\Delta f+\mu_f f&=\chi_{\Theta'}v_f&\text{ in }Q,\\ 
m(\sigma,a,t)=f(\sigma,a,t)&=0&\text{ on }\Sigma, \\
m(x,a,0)=m_0\quad f(x,a,0)=f_0&&\text{ in }Q_A,\\
m(x,0,t)=(1-\gamma)\int_{0}^{A}\beta(a,p)fda,\quad f(x,0,t)=\gamma \int_{0}^{A}\beta(a,p)fda&& \text{ in }Q_T.\\
\end{array}
\right.
\label{bbb1}
\end{equation}	

 By applying Schauder's fixed point Theorem it follows that $N$ admits a fixed point. And therefore proves Theorem 1.2-2.\\
 \begin{remark}
 	Here, we were able to obtain the complete extinction of the females. This makes the possibility to obtain, after an interval of time greater than $A$, the extinction of the entire population since there is no longer birth.
 	But it seems difficult to obtain first the total eradication of male individuals. In our view point, this is naturally difficult because of births.
 \end{remark}
	   \end{document}